\documentclass{article}
\usepackage{amsfonts}
\usepackage{amsmath}
\usepackage{sectsty}
\usepackage{amsthm}
\usepackage{authblk}
\usepackage{blindtext}
\usepackage[british]{babel}
\usepackage{parskip}
\usepackage[a4paper]{geometry}
\usepackage{graphicx}
\usepackage[utf8]{inputenc}
\usepackage{bm}
\usepackage{bbm}
\usepackage[square,numbers]{natbib}
\usepackage{pdflscape}
\usepackage{rotating}
\usepackage{physics}
\usepackage{xcolor}
\usepackage{subcaption}
\usepackage{mathrsfs}
\usepackage{enumerate}

\newcommand{\RR}{\mathbb{R}}

\newcommand{\EE}{\mathbb{E}}

\newcommand{\F}{\mathcal{F}}

\newcommand{\N}{\mathcal{N}}

\DeclareMathOperator{\Var}{Var}

\DeclareMathOperator{\Cov}{Cov}
\DeclareMathOperator{\Corr}{Corr}

\newtheorem{assumption}{Assumption}

\newtheorem{algorithm}{Algorithm} 
\newtheorem{theorem}{Theorem}[section]

\newtheorem{corollary}[theorem]{Corollary}
\newtheorem{lemma}[theorem]{Lemma}
\newtheorem{proposition}[theorem]{Proposition}
\theoremstyle{remark}
\usepackage[colorlinks=true,urlcolor=blue,citecolor=black,linkcolor=blue]{hyperref}

\newtheorem{remarkx}[theorem]{Remark}
\newenvironment{remark}
{\pushQED{\qed}\remarkx}
{\popQED\endremarkx}

\newcommand{\floor}[1]{\lfloor #1 \rfloor}

\newcommand{\diag}{\operatorname{diag}}

\makeindex
\usepackage{titlesec}
\titlelabel{\thetitle.\quad}
\sectionfont{\fontsize{13}{12}\selectfont}
\setcitestyle{semicolon}
\setcitestyle{notesep={; },square,aysep={},yysep={;}}
\usepackage{mathtools}

\title{Strong Invariance Principles
for Ergodic Markov Processes\\}

\author[1]{Ardjen Pengel}
\author[1]{Joris Bierkens}
\affil[1]{\small Delft Institute of Applied Mathematics, Delft University of Technology}
\affil[ ]{\small\textit{E-mail:} \url{ a.l.pengel@tudelft.nl}, \url{ joris.bierkens@tudelft.nl}}

\begin{document}
\date{} 
\maketitle

\setlength{\parskip}{7.5pt}
\setlength\parindent{0pt}
\setcitestyle{semicolon}
\setlength\bibsep{7.5pt}

\begin{abstract} 
\noindent
Strong invariance principles describe the error term of a Brownian approximation of the partial sums of a stochastic process. While these strong approximation results have many applications, the results for continuous-time settings have been limited. In this paper, we obtain strong invariance principles for a broad class of ergodic Markov processes. Strong invariance principles provide a unified framework for analysing commonly used estimators of the asymptotic variance in settings with a dependence structure. We demonstrate how this can be used to analyse the batch means method for simulation output of Piecewise Deterministic Monte Carlo samplers. We also derive a fluctuation result for additive functionals of ergodic diffusions using our strong approximation results.
\end{abstract}

\section{Introduction}
Let $X=(X_k)_{k\in \mathbb{N}}$ be a stochastic sequence defined on a common probability space and consider the partial sum process $S_n$, given by $S_n=\sum_{k=1}^nX_k$. We say that a strong invariance principle holds for $X$ if there exist 
 a probability space
$(\Omega,\mathcal{F},\mathbb{P})$ on which we can construct a sequence of random variables $X'=(X'_k)_{k\in \mathbb{N}}$ and a Brownian motion $W=(W(t))_{ t\geq 0}$, such that $X$ and $X'$ are equal in law and 
$$ \abs{S_n'-\mu n- \sigma W(n)}=\mathscr{O}(\psi_n) \quad \textrm{a.s.,}$$
where $S_n'$ denotes the partial sum process of $X'$, $\mu$ and $\sigma$ are finite constants determined by the law of the process,  $\mathscr{O}$ describes the asymptotic regime, and $\psi_n$ the corresponding approximation error. More specifically, if $S=(S_t)_{t\geq 0}$ denotes a stochastic process and $\psi=(\psi_t)_{t\geq 0}$ is some positive sequence, we write  $$S_T=o(\psi_T)\ \textrm{a.s.}\ \quad \textrm{and}\quad S_T=O(\psi_T)\ \textrm{a.s.}$$
to denote
$$\mathbb{P}\left(\lim_{T\rightarrow \infty}{S_T}\big/{\psi_T}=0\right)=1\ \quad \textrm{and}\quad \mathbb{P}\left(\limsup_{T\rightarrow \infty} \ { \abs{S_T}} \big/
{\psi_T }< \infty \right)=1 $$
respectively. For notational convenience we will usually make no distinction between $X$ and $X'$.

For a sequence of independent, identically distributed random variables with mean zero and unit variance, the Koml\'os-Major-Tusn\'ady approximation \citep{kmt_2,kmt_1} asserts that if $E\abs{X_1}^p<\infty$ for some $p>2$, then  on a suitably enriched probability space, we can construct a Brownian motion $W=\{W(t), t \geq 0\}$ such that
\begin{equation}
\label{kmt1}
S_n=W(n)+o(n^{1/p})\quad \textrm{a.s.}    
\end{equation}
\noindent If we can additionally assume that the moment-generating function exists in an area around zero, i.e., $\mathbb{E}e^{t\abs{X}}<\infty$ for some $t>0$, then
\begin{equation}
\label{kmt_2}
S_n=W(n)+O(\log n)\quad \textrm{a.s.}
\end{equation}
\noindent
Furthermore, if only existence of the the second moment is assumed, \citet{major} showed that there exists a sequence $t_n\sim n$ such that
\begin{equation}
\label{kmt_3}
S_n=W(t_n)+o(n^{1/2})\quad \textrm{a.s.}  
\end{equation}
%
The error terms appearing in the strong invariance principles  (\ref{kmt1}), (\ref{kmt_2}), and (\ref{kmt_3}) are optimal. The approximation error appearing in the strong invariance principle also quantifies the convergence rate in the functional central limit theorem, as shown in  \citet[Theorem 1.16 and Theorem 1.17]{weighted_approx}. These strong approximation results are powerful tools used to obtain numerous results in both probability and statistics as seen in, e.g.,  \citet{applications_sip}, \citet{sip_boek}, \citet{parzen}, and \citet{shorack_empirical}.

Naturally, it is of great interest to extend these results beyond the i.i.d. setting. \citet{berkes} gives an extensive  overview of invariance principles for dependent sequences. 
In Markovian settings, strong approximation results were obtained by \citet{cuny}, \citet{sip_regenerative}, \citet{multivariate_consistency}, and \citet{merlevede2015}, among others. The strong invariance principle of \citet{merlevede2015} attains the Koml\'os-Major-Tusn\'ady bound given in (\ref{kmt_2}). The results of \citet{sip_regenerative} and \citet{merlevede2015} are established through an application of Nummelin splitting, introduced in the seminal papers of \citet{athreya} and \citet{nummelin}. Provided that the transition operator of the chain satisfies a one-step minorisation condition, a bivariate process can be constructed such that this process possesses a recurrent atom and the first coordinate of the constructed process is equal in law to the original Markov chain. Consequently, the chain  inherits a regenerative structure and can thus be divided into independent identically distributed cycles. By application of the Koml\'os-Major-Tusn\'ady approximations strong invariance principles can be obtained. Strong approximation results for Markov chains are useful tools for analysing estimators of the asymptotic variance of Markov Chain Monte Carlo (MCMC) sampling algorithms. \citet{damerdji1991,damerdji1994}, \citet{flegal_bm}, and \citet{multivariate_consistency} show strong consistency of the batch means and spectral variance estimators for MCMC simulation output using the appropriate strong approximation results.

Recently, there has been growing interest in Monte Carlo algorithms based on Piecewise Deterministic Markov Processes (PDMPs). The main appeal of these processes is their non-reversible nature. It is well known that non-reversibility can significantly improve performance of sampling methods, in terms of both convergence rate to equilibrium and asymptotic variance, see for example, the results of \citet{hwang1993} and \citet{lelievre2013} regarding convergence to stationarity and \citet{duncan2016variance}  and \citet{rey2015irreversible} regarding the asymptotic variance. Furthermore, PDMPs have piecewise deterministic paths and can therefore be simulated without discretisation error, in contrast to for example Langevin and Hamiltonian dynamics. 
The primary sampling algorithms belonging to this class are the Zig-Zag sampler and the Bouncy Particle sampler, introduced by \citet{bierkens_intro} and \citet{bouchard2018} respectively. Moreover, since these processes maintain the correct target distribution if sub-sampling is employed, they enjoy advantageous scaling properties to large datasets, as seen in \citet{zigzagsub}.

 In order for many useful results regarding estimation of the asymptotic variance of Markov chain simulation output to carry over to PDMP-based methods, it is required that a strong invariance principle holds for the underlying continuous-time process. In this paper, we obtain strong approximation results for a broad class of (continuous-time) ergodic Markov processes. Firstly, we show that the strong invariance principle given in Theorem \ref{Multi_SIP} can be obtained directly through ergodicity and moment conditions. However, the resulting error rate is not explicit and therefore less convenient to work with. 
 
A natural approach for obtaining a more refined strong invariance principle  would be through regenerative properties of the process. However, it is in general not possible to show that the transition semigroup satisfies a minorisation condition such that a regenerative structure can be obtained. The resolvent chain, on the other hand, does satisfy a one-step minorisation condition. Utilising this result, \citet{locherbach2008num} extend the concept of Nummelin splitting to Harris recurrent Markov processes. Hence we can redefine the process such that it is embedded in a richer process which is endowed with a recurrent atom.
Although the resulting cycles are not independent and we therefore do not have regeneration in the classic sense, we do obtain short range dependence. Therefore we can utilise the approximation results of \citet{SPLIT_SIP} to obtain a strong invariance principle attaining a convergence rate of order $O(T^{1/4}\log T)$. This result is formulated in Theorem \ref{main_sip} and covers a wide range of Markov processes including ergodic diffusions. Although the nearly optimal bound of \citet{SPLIT_SIP} does not carry over, to the best of our knowledge, there are currently no approaches established that lead to superior rates for the  class of processes considered in Theorem \ref{main_sip}. 

For PDMPs we are able to give a strong invariance principle with an improved approximation error. We show that the univariate Zig-zag process has regenerative cycles. This allows us to follow the approach of \citet{merlevede2015} such that the optimal strong approximation error of $O(T^{1/p})$ can be obtained. Moreover, if the target distribution factorises into a product of independent densities, the optimal approximation bound carries over to the multivariate settings. Furthermore, we also show that the results of \citet{merlevede2015} can be extended under less restrictive conditions such that the optimal approximation error (\ref{kmt_2}) is still attained. Finally, we discuss some applications of our obtained strong invariance principles. We demonstrate how the obtained strong approximation results can be utilised for analysing the batch means estimator of the asymptotic variance of continuous-time Monte Carlo samplers. Theorem \ref{t_bm} weakens the existing regularity conditions guaranteeing strong convergence of the batch means estimator in an MCMC setting. This is a direct consequence of the fact that Theorems \ref{zz_sip} and \ref{multi_zz_sip} obtain the optimal approximation rate of $O(T^{1/p})$ whereas previous work on estimation of the MCMC standard error is based on strong invariance principles with limited accuracy, which we further explain in Remark \ref{bm_remark}. Furthermore, we demonstrate the applicability of our results to diffusion processes and show that the magnitude of increments can be described with our obtained approximation results.

This article is organised as follows. In section 2, we give  a brief introduction of Piecewise Deterministic Markov processes and state our motivational example. In Section 3, we review Nummelin splitting in continuous time as introduced by \citet{locherbach2008num} and discuss other relevant results. In Section 4, the main results of the paper are given. In Section 5, we discuss the estimation of the asymptotic variance for PDMC simulation output. Section 6 illustrates the applicability of our results to diffusion processes. In Section 7, the proofs of the main results are given.

\section[Motivating Example: Estimation of the asymptotic variance of 
Piecewise Deterministic Monte Carlo samplers]{\texorpdfstring{Motivating Example: Estimation of the asymptotic variance of 
Piecewise Deterministic Monte Carlo samplers}{Motivating Example: Estimation of the asymptotic variance of \\
Piecewise Deterministic Monte Carlo samplers}}

Suppose our goal is to sample from a probability distribution $\pi(dx)$ on $\mathbb{R}^d$, which admits Lebesgue density 
\begin{equation}
\label{potential}
\pi(x)= \frac{e^{-U(x)}}{\int_{\mathbb{R}^d}e^{-U(x)}\ dx},
\end{equation}
where $U$ is referred to as the associated potential of the target $\pi$. We will assume that $U$ is twice continuously differentiable and can be evaluated pointwise. Typically, the objective is to compute expectations with respect to this distribution, in other words, we are interested interested in $\pi(f)=\int f(x) \pi(dx)$, for some appropriately integrable function $f$. 

Piecewise Deterministic Monte Carlo (PDMC) samplers consist of a position and a velocity component. We will consider processes $Z=(Z_t)_{t\geq0}$ with $Z_t=(X_t,V_t)$, where $X_t$ and $V_t$ denote the position and velocity component respectively. Our process takes values in $E=\mathfrak{X}\times \mathcal{V}$, where $\mathfrak{X}$ denotes the state-space of the position component and $\mathcal{V}$ denotes the space of attainable velocities. Piecewise Deterministic Markov processes are characterised by their deterministic dynamics between random event times along with a Markov kernel that describes the transitions at events. More specifically, their deterministic dynamics are described by some ordinary differential equation. Both the Zig-Zag process and the Bouncy Particle sampler have piecewise linear trajectories characterised by 
 $$\frac{d X_t}{dt}=V_t\  \quad \textrm{and}\  \quad \frac{d V_t}{dt}=0.$$
Thus the rate of change of the position is described by the velocity, whereas the velocity does not change along the deterministic dynamics. Changes in the velocity occur according to some inhomogeneous Poisson rate $\lambda(Z_t).$ The Poisson events consist of changes in the velocity component of our process. The fundamental idea behind these sampling methods is to choose the event rate and the changes in velocity such that the position component explores the state-space according to the target distribution $\pi$. The event rate should increase in an appropriate manner as the position is moving towards regions of lower probability mass.  

For the Zig-Zag process the set of possible velocities is given by $\mathcal{V}=\{-1,+1\}^d$. We can distinguish $d$ types of events for the Zig-Zag sampler. For every dimension $i$ of our position component, an event will consist of flipping component $i$ of the velocity, while keeping the other $(d-1)$ components unchanged. More specifically, our transition at events can be described by $F_i:\mathcal{V}\rightarrow \mathcal{V}$, which is the mapping that flips the $i$-th component of the velocity, i.e., for $v\in \mathcal{V}$ we have that
the $k$-th entry of $F_i(v)$ is given by
\begin{equation*}
\label{flippie}
    (F_i(v))_k=  \left\{
                \begin{array}{ll}
                  -v_k  & \mbox{for}\  k=i
                  \bigskip\\
                 
                   v_k & \mbox{for}\ k \neq i
                \end{array}
              \right.,
\end{equation*}
where $v_k$ denotes the $k$-th entry of the velocity $v$ for $k=1,\cdots,d$. A change in the $i$-th  component of the velocity will be governed by the inhomogeneous Poisson rate $\lambda_i$.
For the (canonical) Zig-Zag sampler these rates are given by
\begin{equation}
\label{zz_intensity}
 \lambda_i(x,v)=\left(v_i \partial_{x_i}U(x)\right)^+,
\end{equation}
where $(x)^+:=\max\{x,0\}$. Hence for the Zig-Zag process events occur with rate
\begin{equation}
\label{zz_intensity_total}
 \lambda_Z(x,v)=\sum_{i=1}^d\lambda_i(x,v)=\sum_{i=1}^d\left(v_i \partial_{x_i}U(x)\right)^+.
\end{equation}
%
%
%
%
%
%
%
%
%
%
%
The simulation scheme for Zig-Zag is given in Algorithm \ref{algo_zigzag} below.
\par\noindent\rule{\textwidth}{1pt}
\vspace{-0.75cm}
\begin{algorithm}
\label{algo_zigzag}
 \normalfont
\textbf{Zig-Zag Sampler} 
\vspace{-0.55cm}
\par\noindent\rule{\textwidth}{1pt}
 \normalfont 1. Initialise $(X_0,V_0)\leftarrow (x,v)$ and  $T_0\leftarrow 0$\\
2. For $k=1,2,\cdots$ simulate $\tau^1_k,\cdots,\tau^d_k$ according to\\
  \normalfont \textcolor{white}{4.} $$\Pr(\tau^i_k\geq t)=\exp\left(-\int_0^t \lambda_i(X_{\tau_{k-1}}+sV_{\tau_{k-1}},V_{\tau_{k-1}})\right)ds,$$
 \normalfont \textcolor{white}{4.} for $i=1,\cdots,d.$\\
3. For $s\in (0,\tau_k)$ set $(X_{\tau_{k-1}+s},V_{\tau_{k-1}+s}) \leftarrow (X_{\tau_{k-1}}+sV_{\tau_{k-1}},V_{\tau_{k-1}})$\\
4. The time of the $k$-th event is given by $T_k=T_{k-1}+i_0$, with $i_0=\min_i\{\tau^i_k\}^d_{i=1}$\\
5. Update velocity of component $i_0$ at the event time\\
 \textcolor{white}{4   }  $V_{T_k}=F_{i_0}(V_{T_{k-1}})$
 \vspace{-0.45cm}
 \par\noindent\rule{\textwidth}{1pt}
\end{algorithm}

In \citet{zigzagsub} it is shown that if we have 
$$\lambda_i(x,v)-\lambda_i(x,F_i(v))=v_i \partial_{x_i}U(x),\ \textrm{for}\ i=1,\cdots,d,$$
\noindent
then the Zig-Zag process has the desired invariant distribution given by $\pi(dx) \nu(dv)$, where  the target distribution $\pi$ is the marginal distribution of the position component and $\nu$ is a uniform distribution over the set of velocities $\mathcal{V}$. Consider the case when the target $\pi$ is of product form, namely $\pi(x)=\prod_{i=1}^d \pi_i(x_i)$, where each $\pi_i$ is a one-dimensional probability
density. Then the Zig-Zag process with stationary distribution $\pi$ can be defined through $d$ independent one-dimensional Zig-Zag processes. The potential of the product form target is given by $U(x)=-\sum_{i=1}^d \log \pi_i(x_i)$, and therefore the corresponding Poisson event rates are given by \begin{equation}
\label{indep_ZZ}
\lambda_i(x,v)=\left(-v_i \frac{\partial_{x_i} \pi_i(x_i)}{\pi_i(x_i)}\right)^+=\left(v_i \partial_{x_i}U(x_i)\right)^+,
\end{equation}
where $U(x_i)=-\log \pi_i(x_i).$ Because the switching intensity of every coordinate only depends on its own position and velocity, we see that the corresponding Poisson processes are independent. Therefore it follows that the $d$-dimensional Zig-Zag process $Z_t$ with target distribution $\pi$ can be decomposed into $d$ independent one-dimensional Zig-Zag processes $(Z_t^i)_{i=1}^d$, where every coordinate $i$ moves according to $Z_t^i$ which has target distribution $\pi_i$ for $i=1,\cdots,d$.

\noindent
For the simulation scheme of the Bouncy Particle Sampler we refer to \citet{bouchard2018}. In the one-dimensional case the canonical BPS and ZZS are described by the same PDMP. For a more detailed introduction to PDMP-based samplers we refer to \citet{pdmp_intro}. It can be shown that under very mild regularity conditions both sampling processes admit a stationary distribution given by
\vspace{-0.15cm}
\begin{align}
\label{eq_dist}
    \mu(dx,dv)=\pi(dx) 	\upsilon(dv),
\end{align}
where the target distribution $\pi$ is the marginal distribution of the position component and $\nu$ is the marginal distribution of the velocity component.
Moreover, an ergodic law of large numbers holds
\vspace{-0.15cm}
\begin{align*}
    \lim_{T\rightarrow \infty}\hat{\mu}_T(f):=\lim_{T\rightarrow \infty} \frac{1}{T}\int_0^Tf(X_s,V_s) \ ds=\int_E f(x,v)\mu(dx,dv)=:\mu(f),
\end{align*}
for all $\mu$-integrable $f$. Given the independence of position and velocity at equilibrium, as seen in (\ref{eq_dist}), the time average of the position component
 is a natural estimator for the space average $\pi(f)$. In order to assess the accuracy of our sampling method, we require a central limit theorem to hold;
  \begin{equation}
 \label{intro_clt1}
     \sqrt{T}\left(\frac{1}{T}\int_0^T f(X_s)ds-\pi(f)\right) \xrightarrow{d} \mathcal{N}(0,\Sigma_f) \ \  \textrm{as}\ T\rightarrow \infty,
 \end{equation}
 \noindent
  and estimate the corresponding asymptotic variance $\Sigma_f$. Moreover, the asymptotic variance is also useful for determining the efficiency of the sampling algorithm via measures such as the effective sample size, see for example \citet{gong2016} or \citet{multivariate_output}. The estimation of the asymptotic variance is also required for the implementation of stopping rules, which consists of justifiable criteria for termination of the simulation. In order to validate stopping rules that guarantee a desired level of precision, \citet{glynn} show that the estimator of the asymptotic covariance matrix must be strongly consistent. Strong invariance principles play a central role in the analysis of estimators of the asymptotic variance of Markov Chain Monte Carlo (MCMC) sampling algorithms, see for example \citet{damerdji1991,damerdji1994}, \citet{flegal_bm}, and \citet{multivariate_consistency}.  In this paper, we obtain strong approximation results for broad classes of ergodic Markov processes. We show that for PDMPs many results regarding estimation of the asymptotic variance immediately carry over.

\section{Nummelin splitting in continuous time} Let $X=(X_t)_{t\geq 0}$ be a stochastic process defined on a filtered probability space $(\Omega, \mathcal{F}, (\F_t)_{t\geq 0}, \mathbb{P}_x )$, with Polish state space $(E,\mathscr{E})$ and initial value $X_0=x.$ We consider the case where $X$ is a positive Harris recurrent strong Markov process with transition semigroup given by $(P_t)_{t\geq0}$ with finite invariant measure $\pi$. By definition of positive Harris recurrence, $\pi$ can be normalised to be a probability measure and we have that
\begin{equation}
\pi(A)>0 \implies \mathbbm{P}_x \left(\int_0^\infty \mathbbm{1}_{\{ X_s \in A \}}ds = \infty \right)=1, \quad x\in E.
\end{equation}
Throughout this paper we will additionally require ergodicity of the considered processes. We say that a Markov process $X$ is \emph{ergodic} with convergence rate $\Psi$ if
\begin{equation}
\label{ergodic}
 \norm{P_t(x,\cdot)-\pi}_{TV}\leq V(x)\Psi(t),\quad \textrm{for all} \ x \in E \ \textrm{and}\ t\geq 0,
\end{equation}
 where $V$ is some positive $\pi$-integrable function and $\Psi$ some positive rate tending to zero. Furthermore, a process is called polynomially or exponentially ergodic if $\Psi$ decays accordingly. For a more thorough discussion of these definitions we refer to \citet{meyn1993}. 

The resolvent chain $\bar{X}=(\bar{X}_n)_{n \geq 0}$ is obtained by observing the process at independent exponential times, i.e., $\bar{X}_n:=X_{T_n}$ for $n\geq0$. Here $(T_n)_{n \geq 0}$ denote the sampling times at which we observe the process $X$, which are defined as $T_0:=0$ and $T_n:=\sum_{k=1}^n \sigma_k$, where $(\sigma_k)_{k\geq 1}$ denote a sequence of i.i.d. exponential random variables. The resolvent chain will inherit positive Harris recurrence from the original process. The transition kernel of the process $\bar{X}=(\bar{X}_n)_{n\in \mathbb{N}}$ is given by 
\begin{equation}
\label{resolvent_kernel}
    U(x,A)= \int_0^\infty P_t(x,A) e^{-t}dt,
\end{equation}
and satisfies the one-step minorisation condition, see for example \citet{hopfner} or  \citet{revuz},
\begin{equation}
\label{min_resolvent}
    U(x,A) \geq h \otimes \nu (x,A),
\end{equation}
where $h \otimes \nu (x,A)=h(x)\nu(A)$, with $h(x)=\alpha \mathbbm{1}_C(x)$ for some $\alpha \in (0,1)$, a measurable set $C$ with $\pi(C)>0$, and $\nu(\cdot)$ a probability measure equivalent to $ \pi(\cdot \cap C)$. 

The minorisation condition of the resolvent chain motivates the introduction of the kernel $K((x,u),dy): E \times [0,1]\rightarrow E$ given by 
\begin{equation}
\label{KK}
    K((x,u),dy)=  \left\{
                \begin{array}{ll}
                  \nu(dy)   & \mbox{for}\  (x,u) \in C \times [0,\alpha]
                  \bigskip\\
                 
                  W(x,dy)  & \mbox{for} \ (x,u) \in C \times (\alpha,1]
                  \bigskip\\
                  
                  U(x,dy)  & \mbox{for}\  x \notin C 
                \end{array}
              \right.,
\end{equation}
where the residual kernel $W(x,dy)$ is defined as 
\begin{equation}
\label{residualkernel}
    W(x,dy)=\frac{U(x,dy)-\alpha \nu(dy)}{1-\alpha}.
\end{equation}
Since the resolvent chain is also positive Harris recurrent, it will hit $C$ infinitely often. Given that the resolvent chain has hit $C$, with probability $\alpha$ the chain will move independently of its past according to the small measure $\nu$ and with probability $(1-\alpha)$ it will move according to the residual kernel $W$. By the Borel-Cantelli lemma the residual chain will generate according to $\nu$ infinitely often. Let $R_k$ denote the $k$-th time that the resolvent chain moves according to $\nu$. The randomised stopping times $(R_k)_k$ serve as regeneration epochs for the resolvent; for every $k$, $\bar{X}_{R_k}$ has law $\nu$ and is independent of both its past and of $R_k$.  The implied regenerative properties that the process $X$ obtains through its resolvent are made explicit with the approach of \citet{locherbach2008num}. 
 Their framework requires the following regularity conditions on the transition semigroup of the process $X$:
 
\begin{assumption} 
\label{AssumptionLL}
\begin{itemize}
    \item[(i)]
The semigroup $(P_t)_{t\geq0}$ is  Feller, i.e., for every $A \in \mathscr{E}$, the mapping $x \longmapsto P_t(x,A)$ is bounded.
\item[(ii)] There exists a $\sigma$-finite  measure $\Lambda$ on $(E,\mathscr{E})$ such that for every $t>0$, $P_t(x,dy)=p_t(x,dy)\Lambda(dy)$, with $(t,x,y)\longmapsto p_t(x,y)$ jointly measurable.
\end{itemize}
\end{assumption}

At the so-called sampling times of the process $X$, we can apply the Nummelin splitting technique to the resolvent chain. We then fill in the original process between the sampling times. Following this procedure, \citet{locherbach2008num} construct on an extended probability space a process $Z$ with state space $E \times [0,1] \times E$, that admits a recurrent atom. The first coordinate of $Z$ has the same law as the original process $X$, the second coordinate denotes the auxiliary variables employed in order to generate draws from the resolvent chain via the splitting procedure, and the third coordinate corresponds to the subsequent values of the resolvent chain. 

The process $Z=(Z^1_t,Z^2_t,Z^3_t)_{t\geq0}$ can be constructed according to the following procedure.
Firstly, let $Z^1_0=X_0=x$. Independently from $Z_1$ generate $Z_0^2\sim U[0,1]$, where $U[0,1]$ denotes the uniform distribution on the unit interval. Given $\{Z_0^2=u\}$, draw $Z_0^3$ according to $K((x,u),dx')$. Then inductively for $n\geq 1$, on $Z_0=(x,u,x')$:
\begin{enumerate}[I.]
    \item Choose $\sigma_{n+1}$ according to
    \begin{equation}
    \label{cond_samplingtime}
    \left(\frac{p_t(x,x')}{u(x,x')}\mathbbm{1}_{\{0<u(x,x')< \infty\}} + \mathbbm{1}_{\{u(x,x') \in \{0, \infty\}\}} \right) e^{-t}dt \ \ \textrm{on}\ \ \mathbb{R}_{+}.\end{equation}
     The next sampling time $T_{n+1}$ is given by $T_n + \sigma_{n+1}$.
    
    \item On $\{\sigma_{n+1}=t\}$, put $Z^2_{T_n+s}:=u$ and $Z^3_{T_n+s}:=x'$ for all $0\leq s < t$.
    \item Draw a bridge of $Z^1$ conditioned on its endpoints $Z^1_{T_n}$ and $Z^1_{T_{n+1}}$, so that for every  $0\leq s < t$ we obtain
    \begin{equation}
    \label{ZZ_bridgepoint1}
    Z^1_{T_n+s} \sim \frac{p_s(x,y)p_{t-s}(y,x')}{p_t(x,x')}\mathbbm{1}_{\{p_t(x,x')>0\}}\Lambda(dy)\end{equation}
    Let $Z^1_{T_n+s}:= x_0$ for some fixed $x_0 \in E$ on ${\{p_t(x,x')=0\}}$. Moreover, given
    $Z^1_{T_n+s}=y$ on $s+u <t$ we have that
    \begin{equation}
    \label{ZZ_bridgepoint2}
    Z^1_{T_n+s+u} \sim \frac{p_u(y,y')p_{t-s-u}(y',x')}{p_{t-s}(y,x')}\mathbbm{1}_{\{p_{t-s}(y,x')>0\}}\Lambda(dy)
    \end{equation}
    Again, on ${\{p_{t-s}(y,x')=0\}}$, let $Z^1_{T_n+s}=x_0.$
    \item At jump time $T_{n+1}$ we have $Z^1_{T_{n+1}}:=Z^3_{T_n}=x'$.
    Draw $Z^2_{T_{n+1}}$ independently of $Z_s, s<T_{n+1}$, uniformly on the unit interval.  Given $\{Z^2_{T_{n+1}}=u'\}$, generate 
    $$Z^3_{T_{n+1}}\sim K((x',u'),dx'').$$
\end{enumerate}
Note that in the construction of $Z$ the inter-sampling times $(\sigma_n)_{n\geq1}$ are drawn according to (\ref{cond_samplingtime}), their conditional distribution given the starting and endpoint of the sampled chain. Equation (\ref{ZZ_bridgepoint1}) and (\ref{ZZ_bridgepoint2}), describe the distributions of points in a bridge of the process $X$. The first coordinate of $Z$ consists of bridges drawn according to the law of the original process $X$, between realisations of the resolvent chain.  The results of \citet{locherbach2008num,locherbach44} that we work with are given in the following propositions. Firstly, the first coordinate of $Z$ has the desired distribution.
\begin{proposition}
\label{prop1}The constructed process $Z$ is a Markov process with respect to its natural filtration $\mathbb{F}$. Moreover, the first coordinate $Z^1$ is equal in law to our process $X$, namely, 
$$\mathcal{L}((X_t)_{t\geq0}|X_0=x)=\mathcal{L}((Z^1_t)_{t\geq0}|Z^1_0=x).$$
Moreover, $(T_n-T_{n-1})_{n\geq 1}$ are i.i.d exponential random variables and are independent of $Z^1$; therefore, we also have that  
$$\mathcal{L}((X_{T_n})_{n\geq0}|X_0=x)=\mathcal{L}((Z^1_{T_n})_{n\geq0}|Z^1_{0}=x).$$
\end{proposition}
Moreover, the process $X$ is embedded in a richer process $Z$, which admits a recurrent atom $A:=C \times [0,\alpha] \times E$ in the sense of the following proposition. 

\begin{proposition}
\label{recurrent_atom}
Let $(S_n,R_n)$ be a sequence of stopping times defined as $S_0=R_0:=0$ and 
$$S_{n+1}:=\inf \{T_m > R_n: Z_{T_m}\in A\} \ \ \textrm{and} \ \ R_{n+1}:=\inf \{T_m: T_m> S_{n+1}\}.$$
Then $Z_{R_n^+}$ is independent of $\mathcal{F}_{R_{n-1}}$ for all $n\geq 1$ and $(Z_{R_n})_{n\geq 1}$ is an i.i.d sequence with $$Z_{R_n}\sim \nu(dx)\lambda(du)K((x,u),dx')\ \ \textrm{for all}\ n\geq 1.$$
\end{proposition}
The stopping times $\{S_n\}_n$ thus denote the hitting times of the recurrent atom $A$ for the jump process $(Z_{T_n})_n$, and $\{R_n\}_n$ denote the implied regeneration epochs of the process $Z.$ As a direct consequence, we obtain the following regenerative structure for the original process.
\begin{proposition}
\label{prop3}Let $f$ be a measurable $\pi$-integrable function, then we can construct a sequence of increasing stopping times $\{R_n\}_n$ with $R_0=0$ and
\label{NOMU}
 $$ \xi_n := \int_{R_{n-1}}^{R_n}f(X_s) \ ds,\ \ n \geq 1,$$
 such that the sequence $\{\xi_n\}_n$ is a stationary  sequence under $\mathbb{P}_\nu.$ Moreover, for $n\geq 2$, $\xi_n$ is independent of $\mathcal{F}_{R_{n-2}}.$
\end{proposition}

The regenerative structure given in Proposition \ref{prop3} was also noted by \citet{sigman}. They define a process $X$ to be \emph{one-dependent regenerative} if there exists, on a possibly enlarged probability space, a sequence of randomised stopping times $R_n$ with corresponding cycle lengths $\rho_n=R_{n+1}-R_n$ such that $\{\left(X_{R_n+t}\right)_{t\geq 0},(\rho_{n+k})_k\}$ has the same distribution for each $n\geq 1$ and 
are independent of $\{(\rho_n)_{n=1}^{k-1}, (X_t)_{t<R_{n-1}}\}$ for $n\geq 2.$
Note that according to this definition the initial cycle is allowed to have a different distribution.\hspace{-0.1cm} \citet{locherbach2008num} give a  constructive approach towards this result, in which they explicitly define the corresponding stopping times and the recurrent atom. By the implied regenerative structure of $X$, we obtain the following characterisation of the stationary measure.

\begin{proposition}[\protect{\citet[Theorem 2]{sigman}}]
\label{ergo_sig}
Let $X$ be a positive recurrent one-dependent regenerative process, then we can characterise its stationary measure as follows
\begin{equation}
\pi(A)=\dfrac{1}{\varrho} \ \mathbb{E}_\nu \int_0^{R_1}\mathbbm{1}_{\{X_s \in A\}}\ ds,
\end{equation}
where $\varrho$ is defined as $\mathbb{E}_\nu R_1.$
Moreover, we have the following erdogic law of large numbers
\begin{equation}
\lim_{T\rightarrow \infty}\frac{1}{T}\int_0^Tf(X_s)ds=\dfrac{1}{\varrho} \ \mathbb{E}_\nu \int_0^{R_1}f(X_s)ds \ \ \textrm{a.s.},
\end{equation}
\end{proposition}
for all $f:\mathbb{R}^d\rightarrow \mathbb{R}^p$ with $\pi(\abs{f})<\infty$.
Note that the normalisation constant of $\pi$ in Proposition \ref{ergo_sig} is finite and non-zero due to the positive Harris recurrence of the process.

\begin{remark}
The framework of \citet{locherbach2008num,locherbach44}
 does not require ergodicity. Moreover, it is important to note that contrary to the classically regenerative case, Proposition \ref{ergo_sig} does not imply convergence in total variation to the stationary measure. For a counterexample see  \citet[Remark 3.2]{sigman}. 
\end{remark}

For our applications we will require ergodicity and hence we must additionally impose this as stated in (\ref{ergodic}). These ergodicity requirements are usually established through Foster–Lyapunov drift conditions; see \citet{down_exponential} and \citet{fort_subgeometric} for exponential and polynomial ergodicity respectively. These results have been applied to several classes of diffusion processes, see for example \citet[Theorem 8.3 and 8.4]{cattiaux} and \citet[Theorem 3.1 and 4.1]{stramer1999}. 

For PDMPs, \citet{ZZ_ergodicity} show exponential ergodicity of the Zig-Zag process for target distributions that have a non-degenerate local maximum and appropriately decaying tails.  In \citet{exp_ergo_bouncy} and \citet{durmus2020} conditions for exponential ergodicity of the Bouncy Particle Sampler are given. Utilising hypocoercivity techniques, \citet{andrieu2021subgeometric} establish polynomial rates of convergence for PDMPs with heavy-tailed stationary distributions. When we are concerned with PDMPs we will require the following regularity conditions on the stationary density:
\begin{assumption}
\label{a_zz}
Assume that $\pi$ is twice continuously differentiable, strictly positive, has a non-degenerate local maximum and $\lim_{\norm{x}\rightarrow \infty} \pi(x)=0.$
\end{assumption}
These regularity conditions are often imposed in order to analyse the ergodic behaviour of PDMPs. Assumption \ref{a_zz} with accompanying conditions on the decay of the tails of the target distribution are used to show various rates of ergodicity.

\section{Main Theorems}
The most straightforward approach for obtaining a strong approximation result for Markov processes would be through ergodicity requirements. \citet{MIX_SIP} show that a multivariate strong invariance principle holds for sums of random vectors satisfying a strong mixing condition; see also Theorem \ref{Koning_SIP}. This mixing condition can easily be satisfied by guaranteeing an appropriate rate of ergodicity of the process. All proofs are provided in Section 7.

\begin{theorem}
\label{Multi_SIP}
Let $X$ be polynomially ergodic of order $\beta \geq (1+\varepsilon)(1+2/\delta)$ for some $\varepsilon,\delta>0.$ Then for every initial distribution and for all $f:E\rightarrow \mathbb{R}^{d}$ with $\pi
(\norm{f}^{2+\delta})< \infty$,  we can construct a process that is equal in law to $X$ together with a standard $p$-dimensional Brownian motion $W=(W(t))_{t\geq0}$ on some probability space such that
 \begin{equation}
 \label{My_SIP}
 \norm{ \int_0^Tf(X_t) \ dt-T\pi(f)-\Sigma_f^{1/2} W(T)}=O(\psi_T) \quad \textrm{a.s.}
 \end{equation}
 with \begin{equation}
 \label{multi_rate}
\psi_T=T^{1/2-\min(\delta/(2\delta+4), \ \lambda  )} \ \textrm{for some}\  \lambda \in (0,1/2),
 \end{equation}
 and positive semi-definite $d\times d$ covariance matrix $\Sigma_f$ given by
  \begin{equation}
  \label{covvie}
 \Sigma_f=\int_0^\infty \Cov_\pi \left(f(X_0),f(X_s)\right) \ ds+ \int_0^\infty \Cov_\pi(f(X_s),f(X_0)) \ ds,
  \end{equation}
  with all entries converging absolutely and integration of matrices defined element-wise.
\end{theorem}
\begin{remark}
The asymptotic covariance matrix $\Sigma_f$ given in Theorem \ref{Multi_SIP} cannot be simplified. Only for the univariate case $(p=1)$ and for reversible processes do we obtain that \begin{equation}
\label{symcov}
\Sigma_f=2\int_0^\infty \Cov_\pi (f(X_0),f(X_s))\ ds.
\end{equation}
As a result of the reversibility, the cross-covariance matrices in (\ref{covvie}) will be symmetric and thus the asymptotic covariance can be expressed as (\ref{symcov}).
\end{remark}

The rate $\psi_T$ appearing in Theorem \ref{Multi_SIP} will depend on the dependence and moment structure of the considered process. For processes admitting higher order moments and having faster decaying levels of dependence the approximation bound $\psi_T$ will tend to infinity at a slower rate. This can be interpreted as the magnitude of the difference between the centred additive functional of the process and the approximating Brownian motion being smaller. Although result (\ref{My_SIP}) has
 useful applications for arbitrary $\lambda \in (0,1/2)$, many refined limit theorems require an explicit remainder term, where more insight is given regarding the impact of the moment and dependence structure on the approximation error. In order to derive a more refined strong invariance principle we will make us of splitting arguments. Following the continuous time Nummelin splitting  technique, as introduced by \citet{locherbach2008num}, it follows that the process can be embedded in a richer process, which admits a recurrent atom. Hence the process can be redefined such that it can be split in identically distributed blocks of strongly mixing random variables. Therefore we can utilise the approximation results for weakly $m$-dependent sequences of \citet{SPLIT_SIP} to obtain a strong invariance principle; see also Theorem \ref{SPLIT_SIP}.

\begin{proposition}
\label{SIP1}Let $X=(X_t)_{t\geq 0}$ be an aperiodic, positive Harris recurrent Markov process for which Assumption \ref{AssumptionLL} is satisfied. Let $f:E \rightarrow \mathbb{R},$ be a given $\pi$-integrable function. Define the sequence of random times $\{ R_n \}_{n=1}^\infty$ and $\{ \xi_n \}_{n=1}^\infty$ as given in Proposition \ref{recurrent_atom} and \ref{prop3}. Moreover, assume that 
\begin{equation}
\label{mom_cond_1}
    \EE_\nu[R_1^q] < \infty \quad \textrm{for some}\ \  q > 2,
\end{equation}
\begin{equation}
\label{mom_cond_2}
    \EE_\nu\left[\left(\int_0^{R_1}f(X_s)ds\right)^{p} \right]  < \infty \quad \textrm{for some}\ \  p > 2.
\end{equation}
\label{SIP2}
\noindent
Then for every initial distribution we can construct a process, on an enriched probability space, that is equal in law to $X$ together with two standard Brownian motions $W_1$ and $W_2$ such that
\begin{align}
\abs{\int_0^Tf(X_s)ds-T \pi(f)-W_1(s_T^2)-W_2(t_T^2)}=O(\psi_T)\ \textrm{a.s.,}
\end{align}
where $\{\sigma^2_T\}$ and $\{\tau^2_T\}$ are non-decreasing sequences with $\sigma^2_T = \frac{\sigma^2_{\xi}}{\varrho}T  + O\left(\frac{T}{\log T}\right)$, $ \tau_T^2 = O\left(\frac{T}{\log T}\right)$ as $T\rightarrow \infty$, and $ \psi_T, \pi(f), \varrho,  \textrm{and } \sigma_\xi $ are defined in equations (\ref{psi_rate}) to (\ref{sigma_xi}).
\end{proposition}
In Proposition \ref{SIP1} we obtain an explicit approximation error. In alignment with expectations, we see that faster convergence to the stationary measure and the existence of higher order moments will result in an improved approximation error. However, the required moment conditions for Proposition \ref{SIP1} stated in  (\ref{mom_cond_1}) and (\ref{mom_cond_2}) are impractical and would be burdensome, if not impossible, to verify directly for most applications. For classically regenerative Markov chains this problem also arises, see the analogous requirements of regenerative simulation given in \citet{mykland} and the strong approximation result of \citet{sip_regenerative}. \citet{hobert} were the first to simplify moment conditions of this form and give practical sufficient conditions for regenerative simulation. More specifically, in their main result they show that
 polynomial or geometric ergodicity and moment conditions with respect to the stationary measure are sufficient to guarantee finiteness of the second moment of a cycle. This result was generalised to higher order cycle moments by \citet{jones_fixed} and \citet{remarks_fixed}; hence simplifying the required conditions of \citet{sip_regenerative}. However, the aforementioned approaches are all for Markov chains satisfying a one-step minorisation condition, i.e., for the classically regenerative setting. Since our setting involves a more complicated 
 reconstruction of the process of interest, the results do not immediately carry over. In Theorem \ref{SIP2}, we show that the cycle moment conditions (\ref{mom_cond_1}) and (\ref{mom_cond_2}) required for Proposition \ref{SIP1} can also be guaranteed with more easily verifiable ergodicity and moment conditions.

\begin{theorem}
\label{moments} Let $X=(X_t)_{t\geq 0}$ be an aperiodic, positive Harris recurrent Markov process for which Assumption \ref{AssumptionLL} is satisfied. Moreover, let $X$ be polynomially ergodic of order $\beta > p/\varepsilon +2$, for some $\varepsilon>0$ then $$\EE_\nu\left[R_1^{\beta-1}\right] < \infty.$$ Moreover, for all measurable $f: E\rightarrow \mathbb{R}$ with $\pi(\abs{f}^{p+\varepsilon})< \infty$ with $p \geq 1$ we have that 
    $$\EE_\nu\left[\left(\int_0^{R_1}f(X_s)ds\right)^{p}  \ \right]  < \infty.$$
\end{theorem}
By combining Proposition \ref{SIP1} and Theorem \ref{moments} we obtain the desired strong invariance principle.
\begin{theorem}
\label{main_sip}
 Let $X=(X_t)_{t\geq 0}$ be an aperiodic, positive Harris recurrent Markov process for which Assumption \ref{AssumptionLL} is satisfied. Moreover, let $X$ be polynomially ergodic of order $\beta > p/\varepsilon+3$, for given $p>2$ and some $\varepsilon>0$. Then for every initial distribution and for all measurable $f: E\rightarrow \mathbb{R}$ with $\pi(|f|^{p+\epsilon})< \infty$ we can, on an enriched probability space, define a process that is equal in law to $X$ and two standard Brownian motions $W_1$ and $W_2$ such that 
\begin{equation}
\label{main_split_sip}
\abs{\int_0^Tf(X_s)ds-T\pi(f)-W_1(\sigma_T^2)-W_2(\tau_T^2)}=O(\psi_T)\ \textrm{a.s.,}
\end{equation}
where $\{\sigma^2_T\}$ and $\{\tau^2_T\}$ are non-decreasing sequences with $\sigma^2_T = \frac{\sigma^2_{\xi}}{\varrho} T + O\left(\frac{T}{\log T}\right) $,  $ \tau_T^2 = O\left(\frac{T}{\log T}\right)$, and
 \begin{align}
   \psi_T&=\max \left\{T^{1/4}\log T, \label{psi_rate} T^{1/p}\log^2(T) \right\} ,\\
   \pi(f)&=\frac{1}{\varrho}\ \mathbb{E}_\nu \int_0^{R_1}f(X_s) \ ds,\\
   \varrho&=\EE_\nu[R_1], \textrm{ and}\\ \sigma_\xi&=\sqrt{\Var_\nu(\xi_1)+2 \Cov_\nu(\xi_1,\xi_2)} \ \label{sigma_xi}. 
  \end{align}
\end{theorem}
\begin{proof}
The assertion follows immediately from Proposition \ref{SIP1} and Theorem  \ref{moments}.
\end{proof}
The appearance of the second Brownian motion in Theorem \ref{main_sip} is inherited from the strong invariance principle of \citet{SPLIT_SIP}. Although we obtain different time perturbations of the Brownian motions, all desired properties carry over. The second Brownian motion appearing in (\ref{main_split_sip}) is of a smaller magnitude, and will therefore be asymptotically negligible in typical applications. Furthermore, even though the two Brownian motions are not independent, their correlation decays over time
\begin{align}
\label{cor_bm}
\Corr \left(W_1(\sigma_t^2),W_2(\tau_s^2)\right)\rightarrow 0, \quad \textrm{as}\  t,s \rightarrow \infty.
\end{align}
Note that the nearly optimal convergence rate of \citet{SPLIT_SIP} does not carry over. Instead we obtain an approximation error that cannot be improved beyond $O(T^{1/4}\log T)$. Obtaining an approximation error superior to $O(T^{1/4}\log T)$ remains an open problem for the class of processes considered in Theorem \ref{main_sip}. A possible approach for attaining a better convergence rate would be to extend to results of \citet{SPLIT_SIP} to a multivariate setting and then follow the approach of \citet{merlevede2015}. 

The univariate Zig-zag process passes every point in its state-space, in particular also the local optima of its target density, an infinite amount of times. This allows us the define regenerative cycles of the process. Therefore we can adapt the approach of \citet{merlevede2015} and obtain the optimal bound of $O(T^{1/p})$ for the strong approximation of the one-dimensional Zig-Zag process.

\begin{theorem}
\label{zz_sip}
 Let $Z=(X_t,V_t)_{t\geq 0}$  be a aperiodic, positive Harris recurrent Zig-zag process with an invariant distribution $\pi \otimes \upsilon,$ where $\pi$ satisfies Assumption \ref{a_zz}. Moreover, let $Z$ be polynomially ergodic of order $\beta > p/\varepsilon+2$, for given $p>2$  and some $\varepsilon \in \left(0, 1\right)$. Then for every initial distribution and for all measurable $f: E\rightarrow \mathbb{R}$ with $\pi(|f|^{p+\epsilon})< \infty$ there exists a Brownian motion $W$ such that
\begin{equation}
\label{main_zz_sip}
\abs{\int_0^Tf(X_s)ds-T\pi(f)- \sigma^2_fW(T)}=O(T^{1/p})\ \textrm{a.s.,}
\end{equation}
where $\sigma^2_f$ can be characterised as (\ref{covvie_zz}).
\end{theorem}

\citet{merlevede2015} obtains a strong invariance principle for one-dimensional Markov chains satisfying a one-step minorization condition by making use of the implied regenerative properties. Note that their approach carries over for any regenerative process. However, they assume that the chain is exponentially ergodic and that the test function $f$ is bounded. The boundedness of $f$ is very restricting for applications in MCMC, since it excludes many interesting examples such as the posterior mean and variance. Theorem \ref{zz_sip} extends their results by only imposing polynomial ergodicity and only a necessary moment condition for the test function.  

Furthermore, we see that if the target distribution is of product form, i.e., satisfies the factorisation $\pi(x)=\prod_{i=1}^d \pi_i(x_i)$, then the optimal bound carries over to the multivariate settings.

\begin{theorem}
\label{multi_zz_sip}
Let $Z=(X_t,V_t)_{t\geq 0}$ be a aperiodic, positive Harris recurrent Zig-zag process with an invariant distribution $\pi \otimes \upsilon$, where $\pi$ is of product form and every $\pi_i$ satisfies Assumption \ref{a_zz}. Moreover, let $Z$ be polynomially ergodic of order $\beta > p/\varepsilon+2$, for given $p>2$ and some $\varepsilon \in \left(0, 1\right)$. Then for every initial distribution and for all $f:E\rightarrow \mathbb{R}^d$ that can be decomposed as $\prod_i f_i(x_i)$ with $\pi
(\norm{f}^{p})< \infty$,  there exists a standard $d$-dimensional Brownian motion $W$ such that
 \begin{equation}
 \norm{ \int_0^Tf(X_t) \ dt-T\pi(f)-\Sigma_f^{1/2} W(T)}=O(T^{1/p}) \quad \textrm{a.s.}
 \end{equation}

 and covariance matrix $\Sigma_f=\diag \{\sigma^2_{f_1},\cdots, \sigma^2_{f_d}\}$ with
  \begin{equation}
  \label{covvie_zz}
 \sigma^2_{f_i}=\int_0^\infty \Cov_\pi (f_i(X^i_0),f_i(X^i_s)) \ ds+ \int_0^\infty \Cov_\pi (f_i(X^i_s),f_i(X^i_0)) \ ds.
  \end{equation}
\end{theorem}
\noindent
Note that although the proof of Theorem \ref{multi_zz_sip} relies on the fact that the $d$-dimensional Zig-Zag process $Z$ can be decomposed into $d$ one-dimensional independent Zig-Zag processes, the multivariate invariance principle does not directly follow from an application of Theorem \ref{zz_sip}, 
since even though the individual coordinates have regenerative cycles, the multivariate process $Z$ does not possess regeneration times. Moreover, it must be guaranteed that the approximating Brownian motions for the individual components are defined on the same probability space.

\begin{remark}
The aforementioned results require a certain degree of polynomial ergodicity, but can mutatis mutandis be seen to hold assuming exponential ergodicity. 
\end{remark}

\section{Analysis of batch means for Piecewise Deterministic Monte Carlo}

In order to assess the accuracy of our PDMC sampler, we require a central limit theorem to hold and estimate the corresponding asymptotic variance. In \citet{bierkens2017limit} several conditions are given to obtain a CLT for the univariate Zig-Zag process.  
\citet{durmus2020}, \citet{exp_ergo_bouncy}, and \citet{ZZ_ergodicity} obtain a CLT for the Bouncy Particale sampler and Zig-Zag process respectively through geometric drift conditions, which in turn also imply exponential ergodicity. The strong invariance principles we obtained in Theorems \ref{Multi_SIP}, \ref{main_sip},  \ref{zz_sip}, and \ref{multi_zz_sip} immediately imply the following central limit theorems for polynomially ergodic Markov processes. 

 \begin{corollary}
 \label{corcltjes}
  Let $(Z_t)_{t \geq 0}$ with $Z_t=(X_t,V_t)$ be polynomially ergodic of order $\beta \geq (1+\varepsilon)(1+2/\delta)$ for some $\varepsilon,\delta>0$ with stationary measure $\mu$ satisfying (\ref{eq_dist}). Then for all $f:E\rightarrow \mathbb{R}^d$ with $\mu
(\norm{f}^{2+\delta})< \infty$, a
  central limit theorem holds:
  \begin{equation}
  \label{clt}
\frac{1}{\sqrt{T}} \int_0^T( f(X_s,V_s) -\mu(f))\ ds  \overset{d}{\longrightarrow} \N_p(0,\Sigma_f). 
\end{equation}
  Additionally, also a functional central limit theorem holds:
\begin{equation}
    \left(\frac{1}{\sqrt{n}}\int_0^{nt}(f(X_s,V_s)-\mu(f))\ ds \right)_{t \in [0,1]} \overset{d}{\longrightarrow} \Sigma_f^{1/2} W \ \textrm{as}\ n \rightarrow \infty,
\end{equation}
where \begin{equation}
  \label{covvie_2}
 \Sigma_f=\int_0^\infty \Cov_\mu (f(X_0,V_0),f(X_s,V_s)) \ ds+ \int_0^\infty \Cov_\mu (f(X_s,V_s),f(X_0,V_0)) \ ds,
  \end{equation}
  $W=(W_t)_{t \in [0,1]}$ denotes a standard $p$-dimensional Brownian motion, and the weak convergence is with respect to the Skorohod topology on $D[0,1]$, the space of real-valued càdlàg functions with domain $[0,1]$.
\end{corollary}
\begin{proof}
By \citet[Theorem 1.E]{philippSIP}, the FCLT immediately follows from the strong invariance principle formulated in Theorem \ref{Multi_SIP}.
Similarly, by \citet[Proposition 2.1]{damerdji1994} the CLT follows.
\end{proof}
By the same argument the CLT follows for the processes considered in Theorems  \ref{main_sip}, \ref{zz_sip}, and \ref{multi_zz_sip}. For simplicity, we will mainly consider the one-dimensional case, i.e. our quantity of interest is given by $\pi(f)$, with $f:\mathbb{R}^p\rightarrow
\mathbb{R}$ a given $\pi$-integrable function. Let the simulation output, which in our case consists of the position component of a PDMP, be given by $(X_{t})_{t\in[0,T]}$. Note that from Corollary \ref{corcltjes} also a (functional) central limit theorem follows for the position component of the process.  We are interested in estimating the asymptotic variance (\ref{covvie_2}); which we will denote with $\sigma^2_f$, when we are not considering the multivariate setting. 

The batch means method divides the obtained sample trajectory of our process into non-overlapping parts. The sample variance of the means of the obtained batches gives rise to a natural estimator for the asymptotic variance. More specifically, we divide our simulation output in $k_T$ batches of length $\ell_T$ such that $k_T=\floor{T/\ell_T}$. We proceed by computing the sample average of each obtained batch; 
\begin{equation}
\label{batch_def}
\bar{Z}_i(\ell_T):=\frac{1}{\ell_T}\int_{(i-1)\ell_T}^{i\ell_T}\hspace{-0.2cm}f(X_s)ds, \quad i=1,\dots,k_T.
\end{equation}
If a functional central limit theorem holds for our process, it follows that the computed means $\bar{Z}_i(\ell_T)$ are asymptotically independent and identically distributed for all fixed amount of batches. Hence, we can heuristically reason that the sample variance of  $(\bar{Z}_i(\ell_T)_{i=1}^{k_T}$  will be close to $\Var(\bar{Z}_i(\ell_T)).$ Moreover, since each  $\bar{Z}_i(\ell_T)$ is also an empirical mean, it is reasonable to expect their variance to be approximately $\sigma_f^2/\ell_T.$
 The batch means estimator of the asymptotic variance is defined by correcting the sample variance of the batch means $(\bar{Z}_i(\ell_T)_{i=1}^{k_T}$ by a factor $\ell_T$, namely
\begin{equation}
\label{bm}
\hat{\sigma}_T^2=\frac{\ell_T}{k_T-1}\sum_{i=1}^{k_T}\left(\bar{Z}_i(\ell_T)-\frac{1}{k_T}\sum_{i=1}^{k_T}\bar{Z}_i(\ell_T)\right)^2.
\end{equation}
Following the framework of \citet{damerdji1994}, we impose the following conditions on the amount of batches and their length.
\begin{assumption} \label{as_var_1} Let the amount of batches $k_T$ and their lengths $\ell_T$ be such that
\begin{enumerate}[i.]
  \item $k_T\rightarrow \infty$ , $\ell_T\rightarrow \infty$, and $\ell_T/T \rightarrow 0$ as $T\rightarrow \infty$,
\item $\ell_T$ and $T/\ell_T$ are both monotonically increasing,
\item there exists a constant $c\geq 1$ such that $\sum_{n=1}^{\infty} k_n^{-c}< \infty$.
 \end{enumerate}
\end{assumption}
The first requirement of Assumption \ref{as_var_1} is a necessary condition for consistency as seen from the results of
\citet{glynn1991}. The second requirement is solely for technical reasons and the third requirement ensures that the amount of the batches grows fast enough; if we choose $\ell_T=T^\alpha$ the requirement holds for all $\alpha \in (0,1)$, since we can choose $c>1/(1-\alpha)$.
\begin{theorem}
\label{t_bm}
Let $Z$ be polynomially ergodic of order $\beta > p/\varepsilon + 2$, for given $p>2$  and some $\varepsilon \in \left(0, 1\right)$ with stationary measure $\mu$ satisfying (\ref{eq_dist}) and let $\pi
(\abs{f}^{p})< \infty$. Assume that Assumption \ref{as_var_1} holds and that 
\begin{equation}
\label{as_bm1}\frac{T^{2/p} }{\ell_T}\log(T) \rightarrow 0, \ \textrm{as}\  T\rightarrow \infty,
\end{equation}
then for every initial distribution  $\hat{\sigma}^{2 }_T\rightarrow \sigma^2_f$ as $T\rightarrow \infty$ with probability 1.
\end{theorem}
\begin{proof}
The result follows from Theorem \ref{Multi_SIP}, \citet[Proposition 3]{jones_fixed}, and \citet[Theorem 3.3]{damerdji1994}.
\end{proof}
\vspace{-7.5pt}
\begin{remark}
\label{bm_remark}
Note that Theorem \ref{t_bm} weakens the currently available regularity conditions guaranteeing strong convergence of the batch means estimator. This is a direct consequence of the fact that Theorems \ref{zz_sip} and \ref{multi_zz_sip} obtain the optimal approximation rate of $O(T^{1/p})$ whereas the results of \citet{jones_fixed} are based upon the strong invariance principle of \citet{sip_regenerative} which attains the rate $O(T^\gamma \log T)$, with $\gamma= \max (1/p,1/4)$. More specifically, for $f$ with $\pi(\abs{f}^p)< \infty$ \citet{jones_fixed} requires $T^{\gamma}\log^3(T)/\ell_T \rightarrow 0$ as $T \rightarrow \infty$. In particular for the case where $p>4$, Theorem \ref{t_bm} is able to significantly weaken the conditions on the required batch length $\ell_T$. As a direct results of the smaller batch lengths, we are able use a higher amount of bathes $k_T$, which results in a smaller variance for the batch means estimator, as seen in Theorem \ref{t_mse}. Note that a similar conclusion holds for the overlapping batch means and spectral variance estimators considered in \citet{flegal_bm}.
\end{remark}
We see from the required assumption (\ref{as_bm1}) that a larger approximation error in the strong invariance principle, which corresponds to higher orders of dependence, results in a larger required batch size $\ell_T$. This is in agreement with the idea behind batching methods; every batch should give a proper representation of the dependence structure of the process. Otherwise a structural bias will be introduced in the estimation procedure. On the other hand, choosing the batch size larger than necessary, will result in a lower amount of batches $k_T$ leading to a higher variance for the estimator. Strong approximations can also be used to characterise the mean squared error and obtain a central limit theorem for the batch means estimator.

\begin{theorem}
\label{t_mse}Let $Z$ be polynomially ergodic of order $\beta > p/\varepsilon +2$, for given $p>2$  and some $\varepsilon \in \left(0, 1\right)$ with stationary measure $\mu$ satisfying (\ref{eq_dist}) and let $\pi
(\abs{f}^{p})< \infty$. Let the initial distribution be given by $\mu$ and assume that Assumption \ref{as_var_1} holds and $\mathbb{E}_\mu C^2< \infty$, where $C$ is defined in (\ref{split_zagC}). Then we have that
\begin{equation}
\label{mse}
    \mathbb{E}_{\mu} \abs{\hat{\sigma}^2_T-\sigma_f^2}^2=2\sigma_f^4 \frac{\ell_T}{T}+O\left(\frac{T^{1/p}}{\sqrt{T}}\log^{\frac{1}{2}} T\right)+ O\left(\ell_T^{-1}T^{2/p}\log T\right),
\end{equation}
 Moreover, if $\ell_T^{-1}T^{1/p}(T\log T)^{1/2}  \rightarrow 0$ as $T \rightarrow \infty$, then we obtain a CLT for the batch means estimator
\begin{equation}
  \sqrt{k_T}(\hat{\sigma}^2_T-\sigma^2_f)\xrightarrow{d} \mathcal{N}(0,2\sigma_f^4) \ \textrm{as}\ T\rightarrow \infty.
\end{equation}
\end{theorem}
\begin{proof}
By the imposed conditions of the process, the strong invariance principle formulated in Theorem \ref{Multi_SIP} holds. The first claim then follows by \citet[Theorem 1 and Lemma 3]{damerdji1995}  and the second by \citet[Proposition 2]{damerdji1995}.
\end{proof}
The first and second term in (\ref{mse}) describe the variance, whereas the third term represents the bias. Note that the second term does not depend on $\ell_T$ and tends to zero. The obtained bounds for the variance are sharp, whereas, the bounds for bias have room for improvement.

In the multivariate setting, where our quantity of interest is given by $\pi(f)$, with $f:\mathbb{R}^p\rightarrow
\mathbb{R}^d$ a given $\pi$-integrable function, the batch means estimator is given by
\begin{equation}
\label{multi_bm_est}
\hat{\Sigma}_T=\frac{\ell_T}{k_T-1}\sum_{i=1}^{k_T}\left(\bar{Z}_i(\ell_T)-\frac{1}{k_T}\sum_{i=1}^{k_T}\bar{Z}_i(\ell_T)\right)\left(\bar{Z}_i(\ell_T)-\frac{1}{k_T}\sum_{i=1}^{k_T}\bar{Z}_i(\ell_T)\right)^T,
\end{equation}
\noindent
where $\bar{Z}_i(\ell_T)$ is defined in (\ref{batch_def}).
Given the strong invariance principle of Theorem \ref{Multi_SIP}, the results of \citet{multivariate_output} for the multivariate batch means estimator immediately carry over.

\begin{theorem}
\label{multi_bm}
Let $Z$ be polynomially ergodic of order $\beta \geq (1+\varepsilon)(1+2/\delta)$ for some $\varepsilon,\delta>0.$ Let $f:E\rightarrow \mathbb{R}^{d}$ with $\pi
(\norm{f}^{2+\delta})< \infty$. Assume that Assumption \ref{as_var_1} holds and that 
\begin{equation}
\label{as_multi_bm}\frac{\psi_T^{2}}{\ell_T}\log(T) \rightarrow 0, \ \textrm{as}\  T\rightarrow \infty,
\end{equation}
with $\psi_T$ defined in (\ref{multi_rate}), 
then for every initial distribution we have that $\hat{\Sigma}_T\rightarrow \Sigma_f$ as $T\rightarrow \infty$ with probability 1.
\end{theorem}
\begin{proof}
 The claim follows from Theorem \ref{Multi_SIP} and \citet[Theorem 2]{multivariate_output}.
\end{proof}

Furthermore, if the target distribution is of product form and we consider the Zig-Zag Sampler, then Theorem \ref{multi_zz_sip} gives a strong invariance principle with an explicit approximation error. Therefore, we can replace condition (\ref{as_multi_bm}) of Theorem \ref{multi_bm} with (\ref{as_bm1}) for every component of the Zig-Zag process. This results in a condition that can more easily be verified.

\subsection{Discussion}
\subsubsection{{Batch size selection for PDMC}}
\citet{glynn1991} show that there exists no consistent estimator of $\sigma_f^2$ with a fixed amount of batches. Hence the amount of batches should explicitly depend on the length of the simulation $T$.
For the standard choice $\ell_T=T^\alpha$ we see that for $\alpha > 1/2p$ we obtain both strong consistency and $L^2$-convergence of the batch means estimator. Theorem \ref{t_mse} suggests that $\alpha^*=(2+p)/2p$ would be optimal in the mean squared error sense.
The well-known results of \citet{chien1997large}, \citet{goldsman1990}, and \citet{song1995optimal} obtain a bound for the bias of order $O(\ell_T^{-1})$, which implies an optimal (in the MSE sense) batch size of ${\ell}^{\diamond}_T \asymp {T}^{1/3}\hspace{-0.1cm}.$ However, the aforementioned results require the sampling process to be stationary, uniformly ergodic, and satisfy moment condition $\pi(f^{12})< \infty$.  
Obtaining the bias term of order $O(\ell_T^{-1})$ for batch means under milder conditions remains an unaddressed problem. Theorem \ref{t_bm} and \ref{t_mse} only require a strong invariance principle, which we have shown holds under polynomial ergodicity; a very reasonable assumption for simulation output. Moreover, these results do not require stationarity and thus hold for every initial distribution.
  Theorem \ref{t_mse} imposes more demanding conditions on $\ell_T$ than aforementioned frameworks, however, it is quite reasonable to let the batch size depend on the dependence structure of the process through $\psi_T$, instead of only on through the constant $\int s\gamma(s) ds$, as is the case in the aforementioned results. Moreover, in practice the performance of batch means methods with batch size ${\ell}^{\diamond}_T$ are often found to be sub-optimal whereas larger batch sizes see better finite sample performance, as noted by for example \citet{flegal_bm}.  
  
  An alternative approach for determining the optimal batch size was given by \citet{chien1988small}, who obtains an optimal batch size of $\tilde{\ell}_T \asymp T^{1/2}$ by minimising the distance between the cumulants of the studentised ergodic average and a standard Gaussian, which suggest that the resulting confidence intervals enjoy improved finite-sample properties. 

\subsubsection{{Asymptotic normality of the batch means estimator}}
We see that given polynomial ergodicity, also the central limit theorem for the batch-means estimator carries over to the PDMC setting. The results of 
 \citet{sherman2002large} require uniform ergodicity and the moment condition $\pi(f^{12})<\infty$  in order to obtain  asymptotic normality of the batch means estimator.
 Since uniform ergodicity is not attainable for most practical problems, less stringent conditions on the rate of ergodicity are desired. Theorem~\ref{t_mse} places more restricting conditions on the batch size and excludes the choice $\ell^{\diamond}_T\asymp T^{1/3}.$  \citet{chakraborty2019} obtain a CLT for the batch means estimator assuming reversibility, stationarity, geometric ergodicity and moment condition $\pi(f^8)<\infty.$ Moreover, the required batch size  must be such that $k_T=o(\ell_T^2)$. Hence their result is also unable to guarantee  asymptotic normality for batch size $\ell^{\diamond}_T.$ We see that Theorem \ref{t_mse} gives more practical conditions for guaranteeing asymptotic normality of the batch-means estimator, in particular, the results are applicable to non-reversible processes. 

\subsubsection{Spectral variance and overlapping batch means estimators for the PDMC standard error}

Analogous to the batch means method, given the strong invariance principle formulated in Theorem \ref{Multi_SIP}, many results for other estimators of the asymptotic variance also carry over. \citet{flegal_bm} gives more convenient alternatives for some of the requirements of the framework given in \citet{damerdji1991}. The results of \citet{flegal_bm} regarding spectral variance and overlapping batch means estimators for MCMC output are thus also applicable for PDMC, with minor adjustments to their assumptions. Note that the assumed minorisation condition and geometric ergodicity of the Markov chain in \citet{flegal_bm} are only imposed such that the strong invariance principle of  \citet{sip_regenerative} holds. Although implementation of spectral variance estimators for continuous-time output might be impractical, 
these estimators are still of theoretical interest. Numerous estimation methods, such as overlapping batch means and certain standardised time series methods, with feasible implementation for PDMC output, can be show to be (asymptotically) equivalent to spectral estimators. Furthermore, we expect the results of \citet{multivariate_consistency} and \citet{liu2021batch} regarding spectral variance and generalised overlapping batch means estimators respectively to remain valid in the continuous-time setting. Hence also the implications for the optimal values of the tuning parameters of these estimation methods for the asymptotic variance remains valid. Lastly, note that our results hold for all sampling algorithms that produce continuous-time output, and are not restricted to the PDMP setting.

\section{Increments of Additive Functionals of Ergodic Markov processes}
Strong approximation results enable various asymptotic properties of Brownian motion to carry over to other stochastic processes. In this section, we show that the strong invariance principle given in Theorem \ref{main_sip} can be used to show that the increments of additive functionals of Markov processes are of the same magnitude as Brownian increments, provided we have sufficient decay of the approximation error. The following theorem describes the magnitude of the fluctuations of Brownian increments over subintervals of length $a_T.$

\begin{theorem}[\protect{{\citet[Theorem 1]{biga_csorgo}}}]

\label{Brownie_fluctie}
Let $W=(W_t)_{t \geq 0}$ denote a Brownian motion, and let $a_T$ be a  positive non-decreasing function of $T$ such that $0<a_T\leq T$ and $T/a_T$ is non-decreasing. Then
\begin{equation}
  \limsup_{T\rightarrow \infty}\sup_{0\leq t\leq T-a_T}\sup_{0\leq u \leq a_T}\beta_T\abs{W_{t+u}-W_t}=1  \quad \textrm{a.s.},
\end{equation}
where 
$$\beta_T=\left(2a_T\left[\log \frac{T}{a_T}+\log\log T \right]\right)^{-1/2}.$$
\end{theorem}
  Taking $a_T=T$ gives the law of iterated logarithm, and for $a_T=c \log T$ with $c>0$, the Erd\"{o}s-R\'enyi law of large numbers for Brownian motion is obtained. This fluctuation result has been extended to other processes such as integrated Brownian motion, fractional Brownian motion, and non-stationary Gaussian processes, see \citet{lilimit}, \citet{fractional_increments} and \citet{ortegasize} respectively. While these fluctuation results are  of independent interest, they are also used as building blocks in applications, such as  proving convergence properties of 
  kernel density estimators, see for example \citet{revesz1982} and \citet{deheuvels}. These fluctuation results are also used for proving almost sure convergence of various estimators of the asymptotic variance in simulation output settings, see the references given in Section 4.1. By the Koml\'os-Major-Tusn\'ady approximation the fluctuation result immediately carries over for i.i.d. sequences satisfying appropriate moment conditions, as seen in \citet[Theorem 3.1.1 and 3.2.1]{sip_boek}.

  In order to describe the fluctuations of additive functionals over an interval of a specified length $a_T$, we require an explicit remainder term for the Brownian approximation, as given in Theorem \ref{main_sip}. 
  However, due to the appearance of the second Brownian motion in this invariance principle and the perturbed time sequences, it is not immediate that the Brownian fluctuation result carries over. \citet{SPLIT_SIP} show that the magnitude of the increments of partial sums of weakly $m$-dependent sequences are indeed given by Theorem \ref{Brownie_fluctie}, due to the smaller scaling of the second Brownian motion. However, in our case the perturbed time sequences are random since they depend on the amount of one-dependent regenerative cycles of the process, hence the desired result does not follow directly by \citet[Theorem 4]{SPLIT_SIP}.

\begin{theorem}
\label{flucto}
 Let $X=(X_t)_{t\geq 0}$ be an aperiodic, positive Harris recurrent Markov process for which Assumption \ref{AssumptionLL} is satisfied. Moreover, let $X$ be polynomially ergodic of order $\beta > 3+ p/\varepsilon$, for given $p>2$ and some $\varepsilon>0$. Consider a function $f: E\rightarrow \mathbb{R}$ with $\pi(f)=0$ and\\ $\pi(|f|^{p+\varepsilon})< \infty$. Let $a_T$ be a given positive non-decreasing function of $T$ such that 
 \begin{enumerate}[i.]
     \item $\displaystyle 0<a_T\leq T$,
     \item $\displaystyle T/a_T$ is non-decreasing,
     \item $\displaystyle a_T$ is regularly varying at $\infty$ with index $\zeta \in (0,1].$
 \end{enumerate}
Suppose that $\beta_T\psi_T=o(1)$, where 
$$\beta_T=\left(2a_T\left[\log \frac{T}{a_T}+\log\log T \right]\right)^{-1/2},$$
and
$$\psi_T=\max \left\{T^{1/2q}\log T, T^{1/p}\log^2(T) \right\}.$$
Then we have that 
\begin{equation}
  \limsup_{T\rightarrow \infty}\sup_{0\leq t\leq T-a_T}\sup_{0\leq u \leq a_T} \beta_T\abs{\int_{t}^{t+u}f(X_s)ds}\leq \frac{\sigma^2_\xi}{\varrho}  \quad \textrm{a.s.}
\end{equation}
\end{theorem}
As noted by \citet{SPLIT_SIP},  the spit invariance principle also implies the distributional version of Theorem \ref{flucto}; with similar adaptations to their argument this would also hold in our case. Since the approximation error $\psi_T$ of Theorem \ref{main_sip} cannot be guaranteed to be smaller than $O(T^{1/4}\log T)$, the fluctuation result given in Theorem \ref{flucto} cannot describe the magnitude of increments over slowly growing time intervals $a_T$.

\subsection{Application to diffusion processes}
Diffusions are an important class of processes for which the strong approximation given in Theorem \ref{main_sip} and the related fluctuation result given in Theorem \ref{flucto} are applicable. Let $X=(X_t)_{t \geq 0}$ denote a one-dimensional diffusion process that is defined as the solution of the following time-homogeneous stochastic differential equation (SDE) 
\begin{equation}
\label{sde}
\vspace{-0.15cm}
 \left\{
    \begin{array}{ll} 
        dX_t= b(X_t)d t+ \sigma(X_t) d W_t\\
        X_0\sim \mu,
    \end{array}
\right.
\end{equation}
where $\mu$ is the initial distribution of the process, $\mathfrak{X}\subseteq \mathbb{R}$ denotes the state-space, $b : \mathfrak{X} \xrightarrow{} \RR$ and $\sigma : \mathfrak{X} \xrightarrow{} \RR$ denote the drift and volatility function respectively, and the process $W$ is a Brownian motion. We assume that all required regularity conditions hold such that the existence and uniqueness of a strong solution of the SDE is guaranteed. For example, we can impose Lipschitz conditions on the drift and volatility of the SDE. For a more detailed explanation, we refer to  \citet{rogerswilliams}. 

For diffusion processes to admit the desired ergodic properties we must impose additional regularity conditions. 
The scale function of a one-dimensional diffusion process is given by
\begin{equation}
\label{scale}s(u)=\int_x^u\exp\left[-2\int_{x}^z\frac{b(y)}{\sigma^2(y)}dy \right]dz\ \textrm{and must satisfy}\lim_{u \rightarrow \pm \infty} s(u)=\pm \infty. \end{equation}
If condition (\ref{scale}) holds it follows that the diffusion is recurrent, the time for the process to return to any bounded subset of its state space is a.s. finite.
The speed density of the diffusion process $m:\mathfrak{X} \rightarrow \RR^+$, given by
$m(u)=\left(s(u)\sigma^2(u)\right)^{-1}\hspace{-0.3cm},$ must be Lebesque integrable for the diffusion to be positive Harris recurrent. For higher dimensional diffusion process \citet{bhattacharya} gives conditions that guarantee positive Harris recurrence. In order for the obtained strong invariance principles given in Theorem \ref{Multi_SIP} and Theorem \ref{main_sip} to hold, we require polynomial or exponential convergence to stationarity. These assumptions are usually obtained by verifying drift conditions for the diffusion processes, see for example \citet[Theorem 8.3 and 8.4]{cattiaux} and \citet[Theorem 3.1 and 4.1]{stramer1999}. 


In order for the strong approximation result in Theorem \ref{main_sip} and the related fluctuation result of Theorem \ref{flucto} to hold, the Nummelin splitting scheme of \citet{locherbach2008num} must be applicable. Therefore we must impose regularity conditions such that Assumption \ref{AssumptionLL} is satisfied, i.e., the transition semigroup of the diffusion must be Feller and admit densities with respect to some dominating measure. Under appropriate growth and continuity conditions on the drift and volatility, diffusion processes are Feller, see for example \citet[Theorem 2.2]{williams2006}. Moreover, if the volatility function $\sigma$ is strictly positive (positive-definite in the multivariate case), the diffusion is elliptic and therefore admits transition densities; \citet[Theorem 3.2.1]{stroock1997}. Hence, Assumption \ref{AssumptionLL} is satisfied. Alternatively, for multivariate diffusions, we can impose the parabolic H\"{o}rmander condition which ensures that the propagation of the noise through the different coordinates is sufficient, such that the transition density exists, see for example \citet[Theorem 38.16]{rogerswilliams}.

\subsection{Discussion and suggestions for further research}
 We see that Theorem \ref{Multi_SIP} and Theorem \ref{main_sip} are applicable to a broad class of diffusions and extend the current results on strong approximations for diffusion processes. \citet{sip_diff1} and \citet{sip_diff2} obtain strong invariance principles for diffusions and a complementary fluctuation result and change point test  respectively. The results of \citet{sip_diff2} yield an explicit approximation error comparable to that of Theorem \ref{main_sip}, but are only applicable to stochastic integrals with respect to Brownian, i.e., diffusion processes with no drift. The results of \citet{sip_diff1} give an implicit approximation error and hold for singular diffusions. The strong invariance principle of \citet{sip_diff1} is not covered by our results since singular diffusions generally do not satisfy the mixing properties required for our framework.

The obtained strong invariance principles offer numerous applications for diffusion processes, see for example \citet{applications_sip} and their given references. Following the approach of \citet[Proposition 2]{SPLIT_SIP}, Theorem \ref{main_sip} can be used to obtain a change-point test for diffusions. If the diffusion process we consider has a drift that enforces mean-reversion, we could construct a test for the existence of a deterministic linear trend over specified time periods. This approach would require continuous-time output of a diffusion process, and is therefore more of theoretical interest. However, it is plausible that the asymptotic behaviour of the change-point test should carry over to the high-frequency setting, where the diffusion is observed discretely and it is assumed the inter-observation times tend to zero. 

\section{Proofs}

\subsection{Theorem \ref{Multi_SIP}}
\citet{MIX_SIP} give a strong invariance principle for mixing random variables. In order to state their result, we first briefly introduce mixing coefficients. Let $\mathscr{A}$ and $\mathscr{B}$ denote two sub $\sigma$-algebras of our probability space. The $\alpha$-mixing coefficients of 
two $\sigma$-algebras quantify their dependence as follows 
$$\alpha(\mathscr{A},\mathscr{B})=\sup \{\Pr(F\cap G)-\Pr(F)\Pr(G): F \in \mathscr{A},\ G \in \mathscr{B}\}.$$
The mixing coefficients of a stochastic process $X$, endowed with its natural filtration, are defined as $\alpha_X(s):=\sup_t \alpha(\mathcal{F}_{-\infty}^{ t},\mathcal{F}_{t+s}^\infty) \ \textrm{for }s>0,$
with $\mathcal{F}_{-\infty}^{t}= \sigma(X_u: u \leq t) $ and $
\mathcal{F}_{t+s}^{\infty}= \sigma(X_u: u \geq t+s )$. The mixing coefficients of a process measure the dependence between events in terms of units of time that they are apart. For a stationary Markov process the mixing coefficients simplify as $\alpha(s)=\alpha(\sigma(X_0),\sigma(X_s))$, as shown in for example \citet[page 118]{Bradley}. 

\begin{theorem} [\protect{\citet[Theorem 4]{MIX_SIP}}] 
\label{Koning_SIP}
Let $\xi=(\xi_k)_{k=1}^\infty$ be a weakly stationary sequence taking values in $\mathbb{R}^p$ with mean zero and $\sup_k\mathbb{E}\norm{\xi_k}^{p}\leq 1$, for some $\delta \in (0,1]$. Moreover, let $\alpha_\xi$ the $\alpha$-mixing coefficients of $\xi$ decay polynomially with rate $n^{-(1+\varepsilon)(1+2/\delta)}$ for some $\varepsilon>0$. Then we can redefine $\xi$ on a new probability space on which we can also construct a $p$-dimensional Brownian motion $W$ with covariance matrix $\Sigma_\xi$, with absolutely converging entries 
$$(\Sigma_\xi)_{ij}=\mathbb{E}[\xi_{i1}\xi_{j1}]+\sum_{k=2}^\infty\mathbb{E}[\xi_{i1}\xi_{jk}]+\sum_{k=2}^\infty\mathbb{E}[\xi_{ik}\xi_{j1}],\  \textrm{for}\  1\leq i,j \leq p,$$
such that 
$$\norm{\sum_{k=1}^n \xi_k-W(n)}=O(n^{1/2-\lambda_{\xi}})\ \textrm{a.s.}$$
for some $\lambda_{\xi} \in (0,1/2)$ depending only on $\varepsilon,\delta$ and $p$.
\label{Philly2}
\end{theorem}

The following lemmata are useful in the proof of Theorem \ref{Multi_SIP}.
\begin{lemma}[\protect{\citet[Theorem F.3.3]{douc2018markov}}]
\label{Z_decay}
Let $X$ be a stationary ergodic Markov process with rate of convergence to stationarity given by $\psi$, then 
$\alpha_X(s)$, the $\alpha$-mixing coefficients of the process $X$, decay according to $\psi$, i.e., for all $s\geq 0$ we have that
$$\alpha_X(s)\leq \pi(V)\Psi(s),$$
where $\Psi$ and $V$ are as stated in (\ref{ergodic}).
\end{lemma}

\begin{lemma}[\citet{Davydov} and \citet{Rio}]
\label{cov_bound}
Let $(\Omega,\mathscr{F},\Pr)$ be a probability space and $\mathscr{A}$ and $\mathscr{B}$ be two sub $\sigma$-algebras and consider random variables $X$ and $Y$ that are measurable with respect to these $\sigma$-algebras respectively. Moreover, assume that $X \in L^p(\Pr)$ and $Y \in L^q(\Pr)$, for some $p,q \geq 1$. Then we can bound their covariance in terms of the $\alpha$-mixing coefficients as follows
$$\lvert\Cov (X,Y)\rvert \leq 8 \alpha\left(\mathscr{A},\mathscr{B}\right)^{1/r}\norm{X}_p\norm{Y}_q,\ \textrm{where}\ p,q,r \in [1,\infty]\  \textrm{and}\ \frac{1}{p}+\frac{1}{q}+\frac{1}{r}=1.$$
\end{lemma}
\noindent
Following a traditional blocking argument it is now straightforward to show that the result of \citet{MIX_SIP} also holds for continuous-time ergodic processes.

\subsubsection{Proof of Theorem \ref{Multi_SIP}}
\begin{proof}
For notational convenience introduce $Y=(Y_t)_{t\geq 0}$, where $Y_t:=f(X_t)-\pi(f)$ for $t\geq 0$, and $\xi=(\xi_k)_{k=1}^n$, with $n:=n_T:=\floor{T}$ and 
$\xi_k:=\int_{k-1}^{k}Y_tdt$ for $k=1,\cdots,n$. Note that $Y_t$ is a $d$-dimensional vector, i.e., $Y_t=(Y_{1t},\cdots,Y_{dt})^\top$ and therefore also each $\xi_k$ is a $p$-dimensional vector, $\xi_k=(\xi_{1k},\cdots,\xi_{dk})^\top.$ Furthermore, by definition $n$ is a function of the sample size $T$, however, for notational convenience we suppress this. Since we are in the setting of Lemma \ref{Z_decay}, $X$ has polynomially decaying $\alpha$-mixing coefficients, which we will denote with $(\alpha_X(s))_{s \geq 0}$. Consequently, we have that $Y$ and 
$\xi$ are both stationary and strongly mixing processes with polynomially decaying $\alpha$-mixing coefficients $(\alpha_Y(s))_{s \geq 0}$ and  $(\alpha_\xi(h))_{h \in \mathbb{N}}$ respectively. This can easily be seen by observing that $\sigma\left(f(X_t)\right)\subseteq \sigma \left(X_t\right)$ and $\sigma\left(\xi_k\right)\subseteq \sigma\left(X_s: s\leq k\right)$.  In order to show that a strong invariance holds for $Y$, we will show that it holds for $\xi$ and determine the growth rate of the corresponding remainder terms. Moment conditions for $\xi$ are directly inherited by the assumed moment conditions for $X$. By an application of Jensen's inequality we see that for $p=2+\delta$ we have that
\begin{align*}
    \pi\left(\norm{\xi_k}^{p}\right)&=\mathbb{E}_\pi\norm{ \int_{k-1}^k Y_sds }^{p}\hspace{-0.2cm}\leq \ \mathbb{E}_\pi \int_{k-1}^k \norm{Y_s}^{p}ds = \pi\left(\norm{f-\pi(f)}^{p}\right) < \infty.
\end{align*}
Therefore, by Theorem \ref{Koning_SIP}, we can redefine $\xi$ on a new probability space on which we can also construct a $d$-dimensional Brownian motion $W$ with covariance matrix $\Sigma_\xi$, with absolutely converging entries 
$$(\Sigma_\xi)_{ij}=\mathbb{E}[\xi_{i1}\xi_{j1}]+\sum_{k=2}^\infty\mathbb{E}[\xi_{i1}\xi_{jk}]+\sum_{k=2}^\infty\mathbb{E}[\xi_{ik}\xi_{j1}],\  \textrm{for}\  1\leq i,j \leq d,$$
such that 
$$\norm{\sum_{k=1}^n \xi_k-W(n)}=O(n^{1/2-\lambda_{\xi}})\ \textrm{a.s.}$$
for some $\lambda_{\xi} \in (0,1/2)$ depending only on $\varepsilon,\delta$ and $d$. The claim follows if we show that 
\begin{align}
\label{TP_SIP}
\norm{\sum_{k=1}^{n}\xi_k-\int_0^T Y_t dt}&=O(T^{1/p} )\  \textrm{a.s. for}\ T\rightarrow \infty, \\
\label{Brownian_Trouble}
\norm{W_T-W_{n}}&=o(T^{1/p} ) \ \textrm{a.s. for}\ T\rightarrow \infty,\ \textrm{and}\\
\label{varvergence}
\Sigma_f&=\Sigma_\xi.
 \end{align}

In order to show that (\ref{TP_SIP}) holds, we note that
\begin{equation}
\label{remainder1}
\norm{\int_0^TY_sds-\sum_{k=1}^{n}\xi_k}= \norm{ \int_{n}^TY_sds}\leq \int_{n}^{n+1}\norm{Y_s}ds.
\end{equation}
By a Borel-Cantelli argument it will follow that
\begin{equation}
\label{BorelCantONE}
\norm{\int_{n}^T Y_s\ ds}=O(T^{1/p})\quad \textrm{a.s. for}\  T \rightarrow \infty.
\end{equation}
Indeed, let $\varepsilon>0$ be given and introduce the event $$A_{n,\varepsilon}=\left\{\int_{n}^{n+1}\norm{Y_s}ds> n^{(1+\varepsilon)/p}\right\}.$$ By Markov's inequality it follows that the introduced sequence of events satisfies 
$$\sum_{n=1}^\infty \Pr\left(A_{n,\varepsilon}\right)\leq \sum_{n=1}^\infty \Pr\left(\int_{n}^{n+1}\norm{Y_s}^{p} ds> n^{1+\varepsilon}\right)\leq \pi(\norm{f-\pi(f)})^{p}\sum_{n=1}^\infty \frac{1}{n^{1+\varepsilon}}<\infty.$$
The Borel-Cantelli lemma states that $\mathbb{P}(\limsup A_{n,\varepsilon})=0$. In order for (\ref{BorelCantONE}) to hold, it is sufficient to note that from a certain point in time the remainder term will always satisfy the asserted bound for almost every sample path, i.e.,
\begin{align*}
\Pr\left(\bigcup_{s \geq 0} \bigcap_{T\geq s}\left\{\norm{\int_n^T Y_u du}\leq T^{1/p}\right\}\right) &\geq \Pr\left(\bigcap_{m=1}^\infty \liminf_{n\rightarrow \infty} A_{n,1/m}^c\right)\\
&=\lim_{m\rightarrow \infty}\Pr\left(\liminf_{n \rightarrow \infty} A_{n,1/m}^c\right)=1,
\end{align*}
where the first inequality follows due to monotonicity and the 
second-last equality by the continuity of measures.  A similar Borel-Cantelli argument also shows that (\ref{Brownian_Trouble}) holds.
Introduce sequence of events $$B_{n,\varepsilon}=\left\{\norm{W(T)-W(n)}> n^{(1+\varepsilon)/\kappa}\right\},$$ for given $\varepsilon>0$ and some $\kappa>p$. Since all moments of $\norm{W(T)-W(n)}$ are finite, we have by Markov's inequality that the introduced sequence of events satisfies 
$$\sum_{n=1}^\infty \Pr\left(B_{n,\varepsilon}\right)\leq \sum_{n=1}^\infty \Pr\left(\norm{B_T-B_{n}}^{\kappa} > n^{1+\varepsilon}\right)\leq \mathbb{E}_{\Pr}\norm{W(T)-W(n)}^\kappa\sum_{n=1}^\infty \frac{1}{n^{1+\varepsilon}}<\infty.$$
By completely analogous reasoning as the previous Borel-Cantelli argument we see that 
$$\norm{W(T)-W(n)}=O(n^{1/\kappa})=o(T^{1/p})\  \textrm{a.s. for}\  T \rightarrow \infty.$$
 Therefore the term (\ref{Brownian_Trouble}) will be asymptotically negligible. Finally, we see that by Lemma \ref{cov_bound} it immediately follows that the asserted asymptotic variance $\Sigma_f$ is finite, i.e., 
all entries 
  \begin{equation}
  \label{cov_terms}
(\Sigma_f)_{ij}=\int_0^\infty \Cov_\pi (f_i(X_0),f_j(X_s)) \ ds+ \int_0^\infty \Cov_\pi (f_i(X_s),f_j(X_0)) \ ds\ \  \textrm{for}\  1\leq i,j \leq d
  \end{equation}
  converge absolutely. Indeed,
\begin{flalign*}
 \int_0^\infty\lvert \Cov_\pi (f_i(X_0),f_j(X_s))\rvert ds &\leq 8 \int_0^\infty \alpha_Y(s)^{\delta/p}\pi(\lvert Y_{i0}\rvert ^{p})^{1/p} \pi(\lvert Y_{js}\rvert ^{p})^{1/p} ds & \\
 &\leq 8 \pi(V)^{\delta/p} \pi(\lvert Y_{i0}\rvert ^{p})^{1/p} \pi(\lvert Y_{j0}\rvert ^{p})^{1/p} \int_0^\infty\hspace{-0.22cm} \alpha_X(s)^{\delta/p}ds < \infty.
 \end{flalign*}
 By definition $\alpha$-mixing sequences are monotonically decreasing and bounded by $1/4$. The second term of (\ref{cov_terms}) is treated similarly. In order to show that $\Sigma_f=\Sigma_\xi$, we will show that all entries are equal. Firstly, we decompose the asymptotic covariance matrix as follows
 \begin{flalign}
 \label{vareq}
&\lim_{T\rightarrow \infty} \Var_\pi \left(\frac{1}{\sqrt{T}}\int_0^T Y_tdt\right)=\lim_{T\rightarrow \infty} \Var_\pi \left(\frac{1}{\sqrt{T}}\left(\int_0^n Y_tdt+\int_n^T Y_tdt\right)\right)=& \nonumber\\
&\lim_{T\rightarrow \infty}\frac{1}{T}\Var_\pi\left(\int_0^n 
  Y_tdt\right) + \lim_{T\rightarrow \infty}\frac{1}{T}\Var_\pi\left(\int_n^t 
  Y_tdt\right)+  \lim_{T\rightarrow \infty}\frac{1}{T} \Cov_\pi\left(\int_0^n 
  Y_tdt,\int_n^t Y_tdt\right)+ \nonumber \\
  &\lim_{T\rightarrow \infty}\frac{1}{T} \Cov_\pi\left(\int_n^t Y_tdt, \int_0^n 
  Y_tdt,\right).
 \end{flalign}
 Let  $\Sigma_{T1}, \Sigma_{T2},\Sigma_{T3}$ and $\Sigma_{T4}$ denote the four terms in (\ref{vareq}). We will show that the entry-wise convergence gives us the desired result. For $  1\leq i,j \leq d$ we obtain the following expressions for the elements of the matrices in (\ref{vareq}):
  \begin{align}
  \label{var1}
(\Sigma_{T1})_{ij}&=\lim_{T\rightarrow \infty}\frac{1}{T}\int_0^n  \int_0^n
  \Cov_\pi(Y_{it},Y_{js})\ dtds \\
    \label{var2}
  (\Sigma_{T2})_{ij}&= \lim_{T\rightarrow \infty} \frac{1}{T} \int_n^T \int_n^T  \Cov_\pi(Y_{it},Y_{js})\ dtds \\
    \label{var3}
  (\Sigma_{T3})_{ij}&=\lim_{T\rightarrow \infty}\frac{1}{T} \int_0^n \int_n^T
    \Cov_\pi(Y_{it},Y_{js})\  dtds  \\  
    \label{var4}
    (\Sigma_{T4})_{ij}&=\lim_{T\rightarrow \infty}\frac{1}{T}   \int_0^n \int_n^T \Cov_\pi(Y_{is},Y_{jt})\ dtds  
 \end{align}
 We see that $(\Sigma_{T1})_{ij}$ tends to the asymptotic variance $(\Sigma_f)_{ij}$, since
 $$\left(\Var_\pi\left(\frac{1}{\sqrt{n}} \sum_{k=1}^n \xi_k \right)\right)_{ij}=\frac{T}{n}\cdot \frac{1}{T} \int_0^n\int_0^n \Cov_\pi(Y_{it},Y_{js})\ dtds.$$
Finally, another application of Lemma \ref{cov_bound} gives us that $(\Sigma_{T2})_{ij},(\Sigma_{T3})_{ij}$ and $(\Sigma_{T4})_{ij}$ tend to zero since
\begin{align*}
\frac{1}{T}\int_n^T \int_n^T \abs{\Cov_\pi(Y_{it},Y_{js})}\ dtds&\leq \frac{1}{T} C_{f,V}\int_n^T s^{-\beta\delta/p}ds=o(1), \\
\frac{1}{T}\int_0^n \int_n^T \abs{\Cov_\pi(Y_{it},Y_{js})}\ dtds &\leq \frac{n}{T} C_{f,V}\int_n^T s^{-\beta\delta/p}\ ds =o(1), \ \\
\frac{1}{T}\int_0^n \int_n^T \abs{\Cov_\pi(Y_{is},Y_{jt})}\ dtds&\leq \frac{n}{T} C_{f,V}\int_n^T s^{-\beta\delta/p} \ ds=o(1),
\end{align*} 
with $C_{f,V}=8 \pi(V)^{\delta/p}\pi(\lvert Y_{i0}\rvert ^{p})^{1/p} \pi(\lvert Y_{j0}\rvert ^{p})^{1/p}< \infty.$\\
Note that we have now shown our result assuming stationarity, i.e., with initial distribution $\pi$. However, by a standard argument this can be extended to every initial distribution. It is well known that for ergodic Markov processes every bounded harmonic function is constant, see for example \citet[Theorem 20.10]{kallenberg}. Hence by following the argument of \citet[Proposition 17.1.6]{the_gospel} it follows that the strong invariance principle holds for every initial distribution.
\end{proof}

\subsection{Proposition \ref{SIP1}}
\citet{SPLIT_SIP} give a strong invariance principle for weakly $m$-dependent processes, which are defined as processes that can be approximated by $m$-dependent processes in the $L^p$-sense, with a sufficiently decaying approximation error (rate function in terminology of \citet{SPLIT_SIP}). Their strong invariance principle, stated in Theorem \ref{SPLIT_SIP}, is obtained through a classical blocking argument for $m$-dependent random variables. By dividing an $m$-dependent sequence into non-overlapping long and short blocks, two sequences of independent random variables are obtained; these can both be approximated by a Brownian motion. Trivially, stationary $m$-dependent processes satisfying appropriate moment conditions fall into their framework. For more details we refer to \citet{SPLIT_SIP}.

\begin{theorem} [\protect{\citet[Theorem 2]{SPLIT_SIP}}] 
\label{SPLIT_SIP}
Let $\xi=(\xi_k)_{k=1}^\infty$ be a centered stationary sequence with $\sup_k\mathbb{E}\abs{\xi_k}^{p}< \infty$, for some $\delta >0$. Moreover, let $\xi$ be weakly $m$-dependent in $L^p$ with an exponentially decaying rate function $\kappa$, i.e.,
$$\kappa(m)\ll \exp(-cm), \quad \textrm{for some }c>0.$$
Then the series 
$$\sigma_\xi^2=\sum_{k=0}^\infty \mathbb{E}\xi_0\xi_k$$
is absolutely convergent, and we can redefine $\xi$ on a new probability space on which we can also construct two standard Brownian motions $W_1$ and $W_2$ such that
$$\abs{\sum_{k=1}^n \xi_k-W_1(s_n^2)-W_2(t_n^2)}=O(n^{1/p}\log^2n)\ \textrm{a.s.,}$$
where $\{s_n\}$ and $\{t_n\}$ are non-decreasing deterministic sequences with $s_n^2= \sigma_\xi^2n+O(n/\log n)$ and $t_n^2=O(n/\log n)$.
\label{splittie}
\end{theorem}
As noted by \citet{SPLIT_SIP} the perturbed time sequences $\{s_n\}$ and $\{t_n\}$ are deterministic and can be explicitly calculated.

\subsubsection{Proof of Proposition \ref{SIP1}}
\begin{proof}

By Proposition \ref{prop1} we see that we can redefine our process such that it is embedded in a richer process $Z$. We will identify $X$ as the first coordinate of the process $Z.$ Following Proposition \ref{recurrent_atom}, we introduce the sequence of stopping times $(S_n,R_n)$ defined as $S_0=R_0:=0$ and 
$$S_{n+1}:=\inf \{T_m > R_n: Z_{T_m}\in A\} \ \ \textrm{and} \ \ R_{n+1}:=\inf \{T_m: T_m> S_{n+1}\}.$$
Then $Z_{R_n^+}$ is independent of $\mathcal{F}_{R_{n-1}}$ for all $n\geq 1$ and $(Z_{R_n})_{n\geq 1}$ is an i.i.d sequence with $$Z_{R_n}\sim \nu(dx)\lambda(du)K((x,u),dx')\ \ \textrm{for all}\ n\geq 1.$$
As a direct consequence, the sequence  $\{\xi_n\}_n$ defined as 
 $$ \xi_n := \int_{R_{n-1}}^{R_n}\{f(X_s)-\pi(f) \}\ ds,\ \ n \geq 1,$$
 is  stationary under $\mathbb{P}_\nu.$ 
Moreover, by Proposition \ref{prop3} for $n\geq 2$, $\xi_n$ is independent of $\mathcal{F}_{R_{n-2}}.$
Let $N(T)$  denote the amount of regenerations of the resolvent chain up to time $T$, namely 
 $$N(T)=\max\{k: R_k \leq T \}.$$
It immediately follows that 
$$\int_0^T\{f(X_s)-\pi(f) \} \ ds=\sum_{k=1}^{N(T)}\xi_k+ \int_{R_{N(T)}}^T\{f(X_s)-\pi(f) \} \ ds.$$
Consequently, we have that 
\begin{equation}
\label{remainder2}
\Bigg|\int_0^T\{f(X_s)-\pi(f) \} \ ds-\sum_{k=1}^{N(T)}\xi_k\Bigg|\leq \int_{R_{N(T)}}^T\abs{f(X_s)-\pi(f)}ds.
\end{equation}
By an analogous Borel-Cantelli argument as given in Theorem \ref{Multi_SIP} for the remainder term defined in (\ref{remainder1}),  we have that
\begin{equation}
\label{BorelCant21}
\int_{R_{N(T)}}^T\abs{f(X_s)-\pi(f)}ds=O(T^{1/p})\quad \textrm{a.s. for}\  T \rightarrow \infty.
\end{equation}
Furthermore, by Proposition \ref{prop3}, the sequence $\{\xi_k\}_{k=1}^\infty$ is a stationary $m$-dependent sequence. By the imposed moment conditions and stationarity, we have by the reasoning given in \citet[Section 3.1]{SPLIT_SIP} that $\{\xi_k\}_{k=1}^\infty$ is also a weakly $m$-dependent process with a rate function $\kappa(m)$ equal to zero for $m \geq 1.$ Hence by Theorem \ref{SPLIT_SIP}, we can redefine $(\xi_k)_k$ on a new probability space on which we can also construct two standard Brownian motions $W_1$ and $W_2$ such that
\begin{equation}
\label{S1}
\abs{\sum_{k=1}^n \xi_k-n\mathbb{E}_\nu \xi_1-W_1(s_n^2)-W_2(t_n^2)}=O(n^{1/p}\log^2n)\ \textrm{a.s.,}
\end{equation}
where $\{s_n\}$ and $\{t_n\}$ are increasing deterministic sequences with $s_n^2 = \sigma_\xi^2 n + O(n/ \log n )$ and $t_n^2 =O(n/ \log n )$. Note that by Proposition \ref{ergo_sig} we have that  \begin{equation*}
\pi(f)=\frac{1}{\varrho}\ \EE_\nu\int_0^{R_1}f(X_s)ds.
 \end{equation*}
Hence 
$$\EE_\nu \xi_1=\EE_\nu \int_0^{R_1} \{f(X_s)-\pi(f) \}\ ds =\varrho \cdot \pi(f-\pi(f))=0. $$
Furthermore, by definition of big $O$ in (\ref{S1}), there exists an almost surely finite random variable $C$ such that for almost all sample paths $\omega$ we have that $\textrm{for all } n\geq N_0\equiv N_0(\omega)$ we have that 
\begin{equation}
\label{split1}
\left.\frac{1}{n^{1/p}\log^2n }\abs{\sum_{k=1}^n \xi_k(\omega)-W_1(s_n^2,\omega)-W_2(t_n^2,\omega)}\right. < C(\omega)
\end{equation}
Since we have that $\mathbb{E}_\nu R_1^q< \infty$, by \citet[Theorem 2.4]{weighted_approx} we can construct a Brownian motion $\tilde{W}$ such that
\begin{equation}
    \abs{N(T)-\frac{T}{\mu}-\frac{\Var_\nu(R_1)}{\varrho^{3/2}}\tilde{W}_T}=o(T^{1/q}),
\end{equation}
\noindent
By the law of iterated logarithm for Brownian motion we obtain 
\begin{equation}
\label{standardarguments}
N(T)=\frac{T}{\varrho}+O(\sqrt{T \log \log T})\quad \textrm{a.s.   } 
\end{equation}
\noindent
 Since $N(T)$ is almost surely increasing and tends to infinity, we have that for almost every sample path $\omega$ there exists a  $T_0\equiv T_0(\omega)$ such that $N(T)(\omega)\geq N_0$ for all $T\geq T_0$. Hence we obtain from (\ref{split1}) that
 \begin{equation}
 \label{split2}
 \left. \limsup_{T\rightarrow \infty} \frac{\left\lvert \sum_{k=1}^{N(T)} \xi_k -W_1(s_{N(T)}^2)-W_2(t_{N(T)}^2)\right\lvert }{{N(T)}^{1/p}\log^2(N(T))}
\right. < C \quad \textrm{a.s. },
\end{equation}
where $s_{N(T)}^2$ and $t_{N(T)}^2$ are almost surely increasing sequences, which given $N(T)$ are deterministic with
\begin{align*}
    s_{N(T)}^2 &= \sigma_\xi^2 N(T) + O(N(T)/ \log N(T) )\\
    t_{N(T)}^2 &=O(N(T)/ \log N(T) ).
\end{align*}
We see that (\ref{split2}) can be reformulated as
\begin{align}
\label{stoch_invariance}
\left\lvert \sum_{k=1}^{N(T)} \xi_k -W_1(s^2_{N(T)})-W_2(t_{N(T)}^2)\right\lvert&= O({N(T)}^{1/p}\log^2N(T))) \quad \textrm{a.s. }\\
&= O({T}^{1/p}\log^2T)) \quad \textrm{a.s. }
\end{align}
Here the second equality follows by  (\ref{standardarguments}). Furthermore, the asymptotic behaviour of $N(T)$ motivates the introduction of $\sigma^2_T, \tau^2_T$ and $\beta_T$ defined as 
\begin{align*}
\sigma^2_T &=s^2_{N(T)}-\sigma^2_{\xi}\beta_T\\
\tau^2_T &= t^2_{N(T)}-\beta_T\\
\beta_T &=N(T)-T/\varrho
\end{align*}
By Theorem \ref{Brownie_fluctie}  (see also Theorem 1.2.1 of \citet{sip_boek}) we see that 
\begin{equation}
\label{browniandifference}
\big\lvert W_1(s^2_{N(T)})-W_1(\sigma^2_T) \big\lvert=O(T^{1/4}\log T) \ \textrm{and}\ \big\lvert W_2(t^2_{N(T)})-W_1(\tau^2_T) \big\lvert=O(T^{1/4}\log T),
\end{equation}
with 
$$\sigma^2_T = \frac{\sigma^2_{\xi}}{\varrho} T + O(T/ \log T)\ \textrm{and} \ \tau_T^2 = O(T/ \log T).$$
Combining results (\ref{BorelCant21}), (\ref{stoch_invariance}), and (\ref{browniandifference}) concludes the proof. By the same arguments given in the proof of Theorem \ref{Multi_SIP}, the strong invariance principle holds for every initial distribution.
\end{proof}
Furthermore, we have by \citet[Proposition 1]{SPLIT_SIP} that
$$\Corr(W_1(s_n),W_2(t_m))\rightarrow 0 \quad \textrm{as}\ m,n\rightarrow \infty.$$
Hence (\ref{cor_bm}) immediately follows.

\subsection{Theorem \ref{moments}}
For this proof, we will rely on the following properties of the resolvent chain. Granted that the process $X$ is aperiodic and positive Harris recurrent, then also the resolvent $\bar{X}$ will inherit these properties, as seen in \citet[Propostion 5.4.5]{the_gospel} and  \citet[Theorem 3.1]{tweedie_genres} respectively. Moreover, by  \citet[Theorem 5.3]{down_exponential}, exponential convergence to stationarity is equivalent for $X$ and $\bar{X}$. 
The split chain of the resolvent in turn  obtains aperiodicity and positive Harris recurrence from $\bar{X}$, as seen in for example \citet{nummelin_boek}. Following a co-de-initialising argument of \citet{roberts_rosenthal}, we see that the split chain inherits the rate of convergence of the resolvent chain. To conclude, we see that the split chain inherits aperiodicity, positive Harris recurrence, and the rate of ergodicity from the process $X$.

Note that by Proposition \ref{prop1} $(Z^1_{T_n},Z^2_{T_n})_n$, the jump chain of the first two coordinates of $Z$, has the same distribution as the split chain of the resolvent. From (\ref{min_resolvent}) and (\ref{KK}) we see that $(Z^1_{T_n},Z^2_{T_n})_n$ is a Markov chain taking values in $E':=E\times [0,1]$ that moves according to the kernel
\vspace{-0.02cm}
\begin{equation}
\label{splitkernel}
U'((x,u),(dx,du))=\nu(dx)\lambda(du)\mathbbm{1}_{\{u\leq\alpha \mathbbm{1}_C(x)\}}+W(x,dx)\lambda(du)\mathbbm{1}_{\{u>\alpha \mathbbm{1}_C(x)\}},
\end{equation}
where $\lambda$ denotes Lebesgue measure on the unit interval. Observe that the kernel of the split chain of the resolvent also satisfies a one-step minorisation condition $U'\geq s \otimes \nu \otimes \lambda,$ i.e.,
\begin{equation}
\label{split_min}
U'((x,u),(dx,du))\geq s(x,u)\nu(dx)\lambda(du),
\end{equation}
where $$s(x,u)=\mathbbm{1}_{\{u\leq\alpha \mathbbm{1}_C(x)\}}.$$
Moreover, the split chain of the resolvent is aperiodic, positive Harris recurrent and inherits the rate of convergence to stationarity from $X$.

\begin{lemma}[\protect{\citet[Lemma 1]{hobert}}]
\label{Upi} Let $(X_t)_{t\geq 0}$ be a positive Harris recurrent Markov process with invariant distribution $\pi$. Then for any  $\pi$-integrable function $g:E^{[0,\infty)} \rightarrow \mathbb{R}$ we have the following inequality holds
\begin{equation}
\EE_\pi \abs{g} \geq c \ \EE_\nu \abs{g}, 
\end{equation}
where $c=\alpha \pi(C) $.
\end{lemma}

\begin{proof}
Since the resolvent chain has the same stationary distribution as the process $X$, i.e., $\pi=\pi U$, the claim follows with the identical argument of \citet[Lemma 1]{hobert}. 
\end{proof}

\subsubsection{Proof of Theorem \ref{moments}}

\begin{proof}
Firstly, by the construction of the randomised stopping times $(S_n)_n$ and $(R_n)_n$ we see that $R_n=S_n+\sigma_{n+1}$, where $\sigma_{n+1}$ has a standard exponential distribution. Hence, by the triangle inequality in $L^q(\pi)$ we only need to show that  $\EE_\pi[{S_1}^q] < \infty$, with $$S_1=\inf \{T_n: Z_{T_n}\in C\times [0,\alpha]\times E\}$$
Let $\bar{Z}=(\bar{Z}_n)_n$ denote the jump chain of the process $Z$, i.e., $\bar{Z}_n=Z_{T_n}$, where the  $(T_n)_n$ denote the jump times. Let  $\bar{X}=(\bar{X}_n)_{n \geq 0}$ again denote the resolvent chain. Let $N_t$ denote the amount of jumps up to time $t$.
Let $\bar{\tau}_A$ denote the hitting time of the recurrent atom for jump chain $\bar{Z}$, i.e.,
\begin{align*}
\bar{\tau}_A:&=\inf\{n\geq 0: \bar{Z}_n \in  A \}=\inf\{n\geq 0: \bar{Z}_n \in  C\times [0,\alpha]\times E \}.
\end{align*} 
For notational convenience we introduce $q:=\beta-1$, note that by the assumed ergodicity assumptions we have that $q>(p+\varepsilon)/\varepsilon.$
 From the relation between the expectation of positive random variables and  tail probabilities we can express the expectation of interest as follows 
\begin{align*}
    \EE_\pi S_1^q &=  \int_0^\infty t^q \mathbb{P}_\pi\left( S_1 > t\right) dt\\
      &=  \int_0^\infty t^q \sum_{m=0}^\infty \mathbb{P}_\pi\left( \bar{\tau}_A > m; N_t=m \right) dt\\
       &= \int_0^\infty t^q \sum_{m=0}^\infty  \underbrace{\mathbb{P}_\pi\left( \bar{\tau}_A > m ; N_t=m ;\bar{Z}_0 \in C\right)}_{=0} + \ \mathbb{P}_\pi\left( \bar{\tau}_A > m ; N_t=m ;\bar{Z}_0 \notin C\right) dt\\
      &= \int_0^\infty t^q \int_{E'}\int_{E'} \sum_{m=0}^\infty \frac{ t^{m-1}}{(m-1)!} e^{- t} \sum_{k=m+1}^\infty \left(U'-\nu \otimes \lambda \otimes s \right)^k(x,dz)\mathbbm{1}_A(z) \pi(dx)dt\\
      &= \int_{E'} \int_{E'}  \sum_{m=0}^\infty \int_0^\infty \frac{ t^{m+q-1}}{(m-1)!} e^{- t}\sum_{k=m+1}^\infty \left(U'-\nu \otimes \lambda \otimes s \right)^k(x,dz)\mathbbm{1}_A(z) dt \pi(dx)\\
            &= \int_{E'} \int_{E'} \sum_{k=1}^\infty \left(U'-\nu \otimes \lambda \otimes s \right)^k(x,dz)\mathbbm{1}_A(z) \sum_{m=0}^{k-1}  \frac{\Gamma(m+q)}{\Gamma(m)}  \pi(dx)\\
             &= \int_{E'} \int_{E'}  \sum_{k=1}^\infty \frac{\Gamma(k+q)}{(q+1)\Gamma(k-1)} \left(U'-\nu \otimes \lambda \otimes s \right)^k(x,dz)s(z)    \pi(dx).
\end{align*}
Note that ${\Gamma(k+q)}/{\Gamma(k-1)}$ can be dominated by some polynomial $\psi(k)$ with a leading term of order $k^{q+1}$.
By \citet[Proposition 1.6]{nummelin_poly} we have that 
\[\int_{E'} \int_{E'}  \sum_{k=0}^\infty \psi(k)  \left(U'-\nu \otimes \lambda \otimes s \right)^k(x,dz)s(z)\pi(dx)< \infty .\]
It follows that $\mathbb{E}_\pi S_1^q < \infty.$

For the second statement of Theorem \ref{moments} we follow the argument of
 \citet[Theorem 2]{remarks_fixed} with some minor adaptations. We give the proof for completion.
\begin{align*}
\left[\mathbb{E}_{\pi}\xi_1^p\right]^{1/p}&\leq\left[\mathbb{E}_{\pi}{\left(\int_0^{R_1}\abs{f(X_s)}ds\right)}^p\right]^{1/p}\\ &=\left[\mathbb{E}_{\pi}\left(\int_0^\infty \abs{f(X_s)}  \mathbbm{1}_{\{ R_1\geq s\}}ds\right)^p\right]^{1/p}\\
&\leq \int_0^\infty\left[\mathbb{E}_{\pi}\left( \abs{f(X_s)}^p  \mathbbm{1}_{\{ R_1\geq s\}}ds\right)\right]^{1/p}\\
&\leq \int_0^\infty\left[\mathbb{E}_{\pi} \abs{f(X_s)}^{p+\varepsilon}\right]^{p/(p+\varepsilon)} \left[\mathbb{E}_{\pi} \mathbbm{1}_{\{ R_1\geq s\}}\right]^{\varepsilon/(p+\varepsilon)}ds\\
&\leq \pi\left(\abs{f}^{p+\varepsilon}\right)^{p/(p+\varepsilon)} \int_0^\infty\left[\mathbb{P}_{\pi} ( R_1\geq s)\right]^{\varepsilon/(p+\varepsilon)}ds\\
&\leq \pi\left(\abs{f}^{p+\varepsilon}\right)^{p/(p+\varepsilon)} (1+\pi(R_1^q)) \int_1^\infty s^{-\varepsilon q/(p+\varepsilon)}ds< \infty.
\end{align*}

Here the inequalities follow by Minkowski's integral inequality, H\"{o}lder's inequality, stationarity, and Markov's inequality. Note that the integral on the last line is finite due to the imposed condition on the rate of polynomial ergodicity since $q=\beta-1> (p+\varepsilon)/\varepsilon$. An application of Lemma \ref{Upi} concludes the proof.
\end{proof}
\begin{remark} For the exponentially erogdic case we would make use of 
\citet[Lemma 2.8]{nummelin1982} which states that for an exponentially ergodic Markov chain there exists an $r>1$ such that
$$\int_{E'} \int_{E'}  \sum_{k=0}^\infty r^k  \left(U'-\nu \otimes \lambda \otimes s \right)^k(x,dy)\mathbbm{1}_C(y)\pi(dx)< \infty .$$
\end{remark}

\subsection{Theorem \ref{zz_sip} and \ref{multi_zz_sip}}

\begin{lemma} [\protect{\citet[Lemma 2.4]{merlevede2015}}] Let $B$ be a standard Brownian motion and $N$ be a Poisson process with intensity $\lambda$, independent of $B.$ Then there exists a standard Brownian motion $W$ that is also independent of $N$ such that 
\label{mer_lemma}
$$\abs{B(n)-\frac{1}{\sqrt{\lambda}} W(N(n))}=O(\log(n))$$
\end{lemma}
\begin{proof}
The claim immediately follows from \citet[Lemma 2.4]{merlevede2015} and a Borel-Cantelli argument.
\end{proof}

\subsubsection{Proof of Theorem \ref{zz_sip}}
\begin{proof}

Let $x_0$ denote the smallest local optimum of the density $\pi$, i.e.,
$$x_0=\min \{x: \pi'(x)=0\}.$$
Since the tails of $\pi$ are diminishing, we must have that $x_0$ is a local maximum.  Moreover, define a set $A$ as follows 
$$A=[-M, x_0] \times \{+1\}.$$
\noindent
Note that for all $(x,v)\in A$ we have that the switching intensity $\lambda(x,v)=0$, since the process is moving toward a higher density region.
Hence the process will move deterministically for time $M$ until the point $x_0\times \{+1\}$ is reached and the probability of a velocity change becomes positive. This motivates the introduction to the stopping times $R_n$ defined as

$$R_0=\inf \{ t \geq 0: (X_t,V_t)= (x_0,1)\}, \ \textrm{and}\ R_n=\inf \{ t \geq R_{n-1}: (X_t,V_t)= (x_0,1)\}.$$

\noindent
By the Markov property, the sequence  $\{\xi_n\}$ defined as 
 $$ \xi_n := \int_{R_{n-1}}^{R_n}\{f(X_s)-\pi(f) \}\ ds,\ \ n \geq 1,$$
 is  i.i.d under $\mathbb{P}_\nu,$ with $\nu$ a Dirac measure at the point $x_0 \times \{+1\}.$ Note that this argument holds for any local optimum by the smoothness assumptions on $\pi$. Note that we also have that $R_n=M+\tau_A$ with $\tau_A$ again denoting the hitting time of set $A.$ Since we have that 
$$\{\tau_A>t\} \subset \bigcup_{m=1}^\infty \{\bar{\tau}_A>m; N_t=m\}, $$
\noindent
where $\bar{\tau}$ again denotes the hitting time of the resolvent chain, we can follow the argument of Theorem \ref{moments} to obtain that
 $$\EE_\nu[R_1^{\beta-1}] < \infty.$$ Moreover, for all measurable $f: E\rightarrow \mathbb{R}$ with $\pi(\abs{f}^{p+\varepsilon})< \infty$ where $p \geq 1$ and $\beta> (p+2\varepsilon)/\varepsilon$, we have that
    $$\EE_\nu\left[\left(\int_0^{R_1}f(X_s)-\pi(f)ds\right)^{p}  \ \right]  < \infty.$$
\noindent
\noindent
Define $(\tau_k)_{k \in \mathbb{N}}$ as $\tau_k=R_k-R_{k-1}$ and let $\varrho$ and $\sigma_\varrho^2$ denote the mean and variance respectively. The sequence of random vectors $(\xi_k,\tau_k)$ are independent and identically distributed. If we choose $\alpha= \Cov_\nu(\xi_1,\tau_1)/\Var_\nu(\tau_1)$, then it immediately follows that $\xi_k-\alpha (\tau_k- \varrho)$ and $\tau_k$ are uncorrelated. 

Applying the multivariate Koml\'os-Major-Tusn\'ady approximation given in \citet[Theorem 1]{kmt_multi} and \citet[Theorem 2.1]{weighted_approx}, there exists two independent Brownian motions $B_1$  and $B_2$ such that 
\begin{equation}
\abs{\sum_{k=1}^n\xi_k-\alpha(\sum_{k=1}^n \tau_k- \varrho)- \tilde{\sigma}B_1}=o\left({\psi_n} \right)      
\label{s1}
\end{equation}

\begin{equation}
\abs{R_n-n\varrho-\sigma_\tau B_2(n)}=o\left(\psi_n \right),
\label{s2}
\end{equation}
with 
\begin{equation}
    \label{rr}
    \psi_n=n^{ \min \left(\frac{1}{\beta-1}   ,   \frac{1}{p}\right)}
\end{equation}
\noindent
Note that in (\ref{s1}) we have that $\EE_\nu \xi_1=0$ by Theorem \ref{ergo_sig} and that $\tilde{\sigma}=\Var_\nu(\xi_1-\alpha(\tau_1-\varrho))$. From the assumptions on the rate of ergodicity we see that the approximation error simplifies to $o(n^{1/p}).$ By \citet[Theorem 1(ii)]{kmt_1}, a Poisson Process $N$ with intensity $\lambda=\varrho^2/\sigma^2_\varrho$ can be constructed from the Brownian motion $B_2$ such that
\begin{align}
    \abs{N(n)-\frac{\varrho}{\gamma}n-\frac{\sigma_\rho}{\gamma}B_2(n)}=O(\log n), \label{PP_sip}
\end{align}
\noindent
wehre $\gamma=\sigma^2_\varrho/\varrho$ and $N$ is constructed increment-wise from $B_2$ in a determinstic way and is therefore also independent of $B_1$. From (\ref{s2}) and (\ref{PP_sip}) it follows that \begin{align}
\abs{R_n-\gamma N(n)}=o(n^{1/p}).
\label{claim1}
\end{align} We claim that it therefore follows that
\begin{align}
   \abs{\int_0^{R_n} \hspace{-0.15cm} f(X_s)-\pi(f)ds - \int_0^{\gamma N(n)} \hspace{-0.25cm} f(X_s)-\pi(f) ds}= o(n^{1/p}) \label{claim2}
\end{align}
Indeed, we have that
\begin{align}
    \abs{\int_0^{R_n} \hspace{-0.15cm} f(X_s)-\pi(f)ds - \int_0^{\gamma N(n)} \hspace{-0.25cm} f(X_s)-\pi(f) ds}&= \abs{\int_{R_n}^{\gamma N(n)} \hspace{-0.25cm} f(X_s)-\pi(f)ds}
\end{align}
\noindent
Since  $\abs{R_n-\gamma N(n)}=o(n^{1/p})$, it follows from the ergodic law of large numbers that
\begin{equation}
 \abs{\frac{1}{n^{\frac{1}{p}}}\int_{R_n}^{\gamma N(n)} \hspace{-0.25cm} f(X_s)-\pi(f)ds} \xrightarrow[]{a.s.}0,
\label{ergo_lln}
\end{equation}
\noindent
and therefore our claim (\ref{claim2}) follows. Combining (\ref{s1}), (\ref{claim1}), and (\ref{claim2}) it follows that
\begin{equation}
\abs{\int_0^{\gamma N(n)} \hspace{-0.25cm} f(X_s)-\pi(f) ds-\alpha \gamma N(n)+ \alpha \varrho n - \tilde{\sigma}B_1(n)}=o\left(n^{\frac{1}{p}} \right)    \label{s3}
\end{equation}
\noindent
Let $(\Gamma_s)_{s \geq 0}$ be defined as
$\Gamma_0:=0$ and $\Gamma_s:=N^{-1}(s)$, the generalized inverse of the Poisson process. Taking $n= \Gamma_{n'}$ in (\ref{s3}) and subsequently making the substitution $n=\gamma n'$, it follows that 
\begin{equation}
\abs{\int_0^{n} f(X_s)-\pi(f) ds-\alpha  n+ {\alpha \varrho } \Gamma_{n/ \gamma} - {\tilde{\sigma}}B_1(\Gamma_{n/ \gamma})}=o\left({\Gamma_n}^{{1} / {p}}\right)=o\left(n^{{1} /{p}} \right),     \label{s4}
\end{equation}
\noindent
where we used the fact that $\Gamma$ is a non-decreasing process that tends to infinity. Moreover, since $\Gamma_n$ has a Gamma distribution it follows from the Koml\'os-Major-Tusn\'ady approximation \citep[Theorem 1] {kmt_1} that there exists a Brownian motion $B_3$ such that
\begin{equation}
    \abs{\Gamma_n-\frac{n}{\lambda}-\frac{1}{\lambda} B_3(n)}=O(\log n).
    \label{s5}
\end{equation}
\noindent
Note that the Poisson process $N$ and therefore its corresponding event time process $\Gamma$ are independent of $B_1.$ Therefore by an application of Lemma \ref{mer_lemma} with $n=\Gamma_n$ it follows that there exists a standard Brownian motion $B_4$ independent of $N$ and $\Gamma$ such that 
\begin{equation}
\abs{B_1(\Gamma_n)-\frac{1}{\sqrt{\lambda}} B_4(n)}=O(\log n)
\label{s6}
\end{equation}
\noindent
Applying the obtained approximations given in (\ref{s5}) and (\ref{s6}) to (\ref{s4}) it follows that
\begin{equation}
\abs{\int_0^{n} f(X_s)-\pi(f) ds-\left(\frac{\tilde{\sigma}}{\sqrt{\lambda \gamma}}B_4(t)- \frac{\alpha \varrho}{\lambda \sqrt{\gamma}}B_3(t)\right)} = o\left(n^{{1} /{p}} \right).
\end{equation}
\noindent
Note that since $B_3$ and $B_4$ are independent we have that 
\begin{equation}
W_t=\frac{1}{\sigma_f}\left(\frac{\tilde{\sigma}}{\sqrt{\lambda \gamma}}B_4(t)- \frac{\alpha \varrho}{\lambda \sqrt{\gamma}}B_3(t)\right)
\end{equation}
\noindent
is a standard Brownian motion since
\begin{equation}
\frac{\tilde{\sigma}^2}{\gamma \lambda}+  \frac{\alpha^2 \varrho^2}{\gamma \lambda^2}=\frac{\mathbb{E}_\nu \xi_1^2}{\varrho}=\sigma^2_f.
\end{equation}
\noindent
 Furthermore, by definition of big $O$, there exists an almost surely finite random variable $C$ such that for almost all sample paths $\omega$ we have that $\textrm{for all } n\geq N_0\equiv N_0(\omega)$ we have that 
\begin{equation}
\label{split_zagC}
\left.\frac{1}{n^{1/p} }\abs{\int_0^n f(X_s(\omega))ds-T\pi(f)-\sigma^2_f W_n(\omega)}\right. < C(\omega).
\end{equation}
It immediately follows that (\ref{split_zagC}) also holds for $T$ sufficiently large and hence carries over for $T \rightarrow \infty$.  By the same argument given in the proof of Theorem \ref{Multi_SIP}, the strong invariance principle holds for every initial distribution.
\end{proof}
\subsubsection{Proof of Theorem \ref{multi_zz_sip}}
\begin{proof}
From \citet[Proposition 2.8]{zigzagsub} we see that the Zig-Zag process with a stationary distribution of product form $\pi(x)=\prod_{i=1}^d\pi_i(x_i)$ can be decomposed into $d$ independent Zig-Zag processes, each with stationary distribution $\pi_i$. Since we have that
\begin{equation}
\label{ms1}
   \norm{ \int_0^Tf(X_t) \ dt-T\pi(f)-\Sigma_f^{1/2} W(T)} \leq \sqrt{d} \max_{i} \abs{\int_0^Tf_i(X^i_t) \ dt-T\pi_i(f_i)-\sigma_{f_i}W^i(T)},
\end{equation}
the theorem follows if we can show that a strong invariance principle holds for every component on the same probability space.  Firstly, assume that the initial distribution of $Z$ is $\pi$. 

In order to obtain a Brownian approximation for every coordinate we will use a regenerative argument along the lines of Theorem \ref{zz_sip}.  For every component $i=1,\cdots,d$ we define the following:  $x_0^i$ the smallest local maximum of the density $\pi_i$, i.e.,
$x^i_0=\min \{x: \pi_i'(x)=0\}$ and corresponding set
set
$A_i=[-M, x^i_0] \times \{+1\},$
\noindent
and the sequences of stopping times $\{R^i_n\}_{n\in \mathbb{N}}$ as follows
$$R^i_0=\inf \{ t \geq 0: (X^i_t,V^i_t)= (x^i_0,1)\}, \ \textrm{and}\ R^i_n=\inf \{ t \geq R_{n-1}: (X^i_t,V^i_t)= (x^i_0,1)\}.$$
\noindent
 Furthermore, we also introduce for every coordinate $i$ the sequence  $\{\xi^i_n\}$ defined as 
 $$ \xi^i_n := \int_{R^i_{n-1}}^{R^i_n}\{f(X_s)-\pi(f) \}\ ds,\ \ n \geq 1.$$
 Note that for all components $ \{\xi^i_n\}_n$ is i.i.d under $\mathbb{P}_{\nu_i},$ with $\nu_i$ a Dirac measure at the point $x^i_0 \times \{+1\}.$ 
\noindent
We can follow the argument of Theorem \ref{moments} to obtain that
 $$\EE_{\nu_i}\left[{R^i_1}^{\beta-1}\right] < \infty \quad \textrm{for} \ i=1,\cdots,d.$$ Moreover, for all measurable $f: E\rightarrow \mathbb{R}$ with $\pi(\abs{f}^{p+\varepsilon})< \infty$ where $p \geq 1$ and $\beta > 2+ p/\varepsilon$, we have that 
    $$\EE_{\nu_i}\left[\left(\int_0^{R^i_1}f_i(X^i_s)-\pi(f)ds\right)^{p}  \ \right]  < \infty \quad \textrm{for} \ i=1,\cdots,d.$$

\noindent
\noindent
Note that for the RHS of (\ref{ms1}), we have for every coordinate $i$ that 
\begin{align}
\abs{\int_0^T  f_i(X^i_t) \ dt-T\pi_i(f_i)-\sigma_{f_i}W^i(T)} &\leq \abs{\int_0^{R_1^i}  f_i(X^i_t)- \pi_i(f_i)dt} \\
& + \abs{\int_{R_1^i}^{T}f_i(X^i_t)- \pi_i(f_i)dt \label{ms2} -\sigma_{f_i}W^i(T)}.
\end{align}
By assuming that the process starts at its stationary distribution, it follows by the argument in the proof of Theorem \ref{moments} that $\abs{\int_0^{R_1^i}f_i(X^i_t)- \pi_i(f_i)dt}$ is almost surely finite and hence asymptotically negligible.

Define $(\tau^i_k)_{k \in \mathbb{N}}$ as $\tau_k=R^i_k-R^i_{k-1}$ and let $\varrho_i$ and $\sigma_{\varrho_i}^2$ denote the mean and variance respectively. The sequence of random vectors $(\xi^i_k,\tau^i_k)$ are independent and identically distributed. If we choose $\alpha_i= \Cov_\nu(\xi^i_1,\tau^i_1)/\Var_\nu(\tau^i_1)$, then it immediately follows that $\xi^i_k-\alpha_i (\tau^i_k- \varrho_i)$ and $\tau^i_k$ are uncorrelated. By applying the multivariate Koml\'os-Major-Tusn\'ady approximation given in \citet[Theorem 1]{kmt_multi} and \citet[Theorem 2.1]{weighted_approx} to the sequence of random vectors
$$z_k=(z_k^1,\cdots,z_k^d)^T=((\xi_k^1-\alpha_1(\tau_k^1-\varrho_1),\tau_k^1), \cdots, (\xi_k^d-\alpha_d(\tau_k^d-\varrho_d),\tau_k^d))^T,$$
it follwos that there exists a $2d$-dimensional Brownian motion such that
\begin{align}
    \abs{\sum_{k=1}^n z_k -\mathbb{E}_\nu z_1 -\tilde{\Sigma}_z B_n} = o\left(n^{{1} /{p}} \right),
\end{align}
where $\tilde{\Sigma}_z=\diag(\Var_\nu(z_1),\cdots, \Var_\nu(z_k)).$ All components of $z_k$ are independent and therefore also the corresponding components of the Brownian motion are independent. Note that we have that for every component $z_k^i$ of $z_k$ we have that there exists two independent Brownian motions $B_1$  and $B_2$ such that 
\begin{equation}
\abs{\sum_{k=1}^n\xi^i_k-\alpha_i\left(\sum_{k=1}^n \tau^i_k- \varrho_i\right)- \tilde{\sigma}_iB_{i1}}= o\left(n^{{1} /{p}} \right)      
\label{ss1}
\end{equation}
\begin{equation}
\abs{R^i_n-n\varrho_i-\sigma_{\tau_i }B_{i2}(n)}=o\left(n^{{1} /{p}} \right)
\label{ss2}
\end{equation}
\noindent
Note that in (\ref{ss1}) we have that $\EE_\nu \xi^i_1=0$ by Theorem \ref{ergo_sig} and that $\tilde{\sigma}_i=\Var_\nu(\xi^i_1-\alpha_i(\tau^i_1-\varrho_i))$. 
By following the argument of the proof of Theorem \ref{zz_sip} for every component, we see that
\begin{equation}
\label{fin}
\abs{\int_{R_1^i}^n f_i(X_t^i)-\pi_i(f_i)ds- \sigma_{f_i}W^i_n} = o\left(n^{{1} /{p}} \right) \quad \textrm{for} \ i=1,\cdots,d..
\end{equation}
By combining (\ref{ms1}), (\ref{ms2}) and (\ref{fin}) the claim follows. By the argument given in the proof of Theorem \ref{Multi_SIP}, the strong invariance principle holds for every initial distribution.
\end{proof}

\subsection{Proof of Theorem \ref{flucto}}
\begin{proof}
Firstly, by Theorem \ref{SIP2} 
there exist two standard Brownian motions $W_1$ and $W_2$ such that 
$$\abs{\int_0^Tf(X_s)ds-W_1(\sigma^2_T)-W_2(\tau^2_T)}=O(\psi_T)\ \textrm{a.s.,}$$
where $\{\sigma^2_T\}$ and $\{\tau^2_T\}$ are non-decreasing sequences with $$\sigma^2_T = \frac{\sigma^2_{\xi}}{\varrho} T + O(T/ \log T)\ \textrm{and} \ \tau_T^2 = O(T/ \log T).$$
  as $T\rightarrow \infty$, where
$   \sigma^2_\xi \textrm{ and}\  \varrho $ are  defined in Theorem \ref{SIP2}. An application of our strong invariance principle gives the following
\begin{align*}
   &\limsup_{T\rightarrow \infty}\max_{0\leq t\leq T-a_T}\max_{0\leq u \leq a_T} \beta_T\abs{\int_{0}^{t+u}f(X_u)du- \int_{0}^{t}f(X_u)du} \\  &\leq \limsup_{T\rightarrow \infty}\max_{0\leq t\leq T-a_T}\max_{0\leq u \leq a_T} \beta_T\abs{W_1(\sigma^2_{t+u})-W_1(\sigma^2_{t})}    \\
  & +
   \limsup_{T\rightarrow \infty}\max_{0\leq t\leq T-a_T}\max_{0\leq u \leq a_T} \beta_T\abs{W_2(\tau^2_{t+u})-W_2(\tau^2_{t})}\\
   & + \beta_T O(\psi_T)   \\
      & =: A_1+A_2+A_3.  
\end{align*}
  Since $\beta_T\psi_T=o(1)$, it immediately follows that $\limsup_T A_3=0$ almost surely.
  In order to use the arguments of \citet[Theorem 4]{SPLIT_SIP} for the terms $A_1$ ad $A_2$, we require the following properties of the sequence $\sigma^2_T$; for any $\varepsilon>0$ there exists some $T_0$ such that for all $T\geq T_0$ 
\begin{equation}
\label{time_prop}
 \sigma^2_T\leq \left(\frac{\sigma^2_\xi}{\varrho}+\varepsilon\right)T
 \quad \textrm{and} \quad 
 \sup_{u\geq 0}\{\sigma^2_{u+a_T}-\sigma^2_u\} \leq \left(\frac{\sigma^2_\xi}{\varrho}+\varepsilon\right)a_T.   
\end{equation}
From the proof of Proposition \ref{SIP1} we see that the process $\{\xi_n\}_{n \in \mathbb{N}}$ with $\xi_n:=\int_{R_{k-1}}^{R_k}f(X_s) \ ds$, where $\{R_k\}_{k }$ denotes the regeneration times of the resolvent, satisfies the following strong approximation result 
$$\abs{\sum_{k=1}^n \xi_k-W_1(s_n^2)-W_2(t_n^2)}=O(n^{1/p}\log^2n)\ \textrm{a.s.,}$$
where $\{s_n\}$ and $\{t_n\}$ are increasing deterministic sequences with $s_n^2=\sigma_\xi^2n+O(n/\log n)$ and $t_n^2\sim \sigma_\xi^2n/\log n$.
By the properties of the sequences $s^2_n$ and $t^2_n$ described in Theorem \ref{SPLIT_SIP} we have that  for any $\varepsilon>0$ there exists some $n_0$ such that  for all $n \geq n_0$ 
$$s^2_{n} \leq(\sigma^2_\xi+\varepsilon)n \ \quad  \textrm{and}\  \quad \sup_k \{s^2_{k+{a_n}}- s^2_k\} \leq(\sigma^2_\xi+\varepsilon)a_n $$
Since $\sigma^2_T=\frac{\sigma^2_\xi}{\mu}T+O(T/\log(T))$, the first required property described in (\ref{time_prop}) can be easily seen to hold. Since  $\limsup_{k\rightarrow \infty }( s^2_{k+1}-s^2_{k})=\sigma^2_{\xi}$ and $ N(u)$ tends to infinity almost surely as $u\rightarrow \infty$, it immediately follows that $$\limsup_{u\rightarrow \infty }( s^2_{N(u)+1}-s^2_{N(u)})=\sigma^2_{\xi}.$$
Hence for $u$ sufficiently large we have that $$ s^2_{N(u+a_T)}-s^2_{N(u)}= \hspace{-0.5cm} \sum_{j=1}^{N(u+a_T)-N(u)}\hspace{-0.6cm}s^2_{N(u)+j}-s^2_{N(u)+j-1}\leq (\sigma^2_\xi+\varepsilon)(N(u+a_T)-N(u))$$
 Since $\sigma^2_T$ was defined as $
      \sigma^2_T:=s^2_{N(T)}-\sigma^2_{\xi}\beta_T
      $ with $\beta_T:=N(T)-T/\varrho$ (see proof of Proposition \ref{SIP1}),  we have that for all $\varepsilon>0$ and $u$ sufficiently large
  \begin{align*}
    \sigma^2_{u+a_T}-\sigma^2_{u}&=s^2_{N(u+a_T)}-s^2_{N(u)}-
    \sigma^2_\xi(\beta_{u+a_T}-\beta_{u})\\[-5pt]
    &\leq (\sigma^2_\xi+\varepsilon)(N(u+a_T)-N(u))-\sigma^2_\xi (N(u+a_T)-N(u) )+\frac{\sigma^2_\xi}{\varrho}a_T\\[-7.5pt]
    &= \varepsilon(N(u+a_T)-N(u))+\frac{\sigma^2_\xi}{\varrho}a_T. 
  \end{align*}
Since $a_T\rightarrow \infty$ as $T\rightarrow \infty$ by assumption, and we have that the amount of regenerations of the resolvent chain in an interval of length $a_T$ will be equal to $a_T/\varrho+o(a_T)$, it follows that for all $\varepsilon>0$ there exists some $T_0$ such that for all $T\geq T_0$
\begin{align*}
    \sup_{u\geq 0}\{\sigma^2_{u+a_T}-\sigma^2_{u}\}\leq \left(\frac{\sigma^2_\xi}{\varrho}+2\varepsilon \right)a_T.
\end{align*}
Consequently, we have also shown that the required properties given in (\ref{time_prop}) hold. Hence for $T\geq T_0$ we obtain
\begin{align*}
\max_{0\leq t\leq T-a_T}\max_{0\leq u \leq a_T} \beta_T\abs{W_1(\sigma^2_{t+u})-W_1(\sigma^2_{t})}&\leq \sup_{0\leq t\leq \sigma^2_{T-a_T}}\sup_{0\leq u \leq (\sigma^2_\xi/ \varrho+\varepsilon)a_T} \beta_T\abs{W_1(t+u)-W_1(t)}\\
 &\leq \hspace{-8pt} \sup_{0\leq t\leq (\sigma^2_\xi/ \varrho+\varepsilon)(T-a_T)} \hspace{-2pt} \sup_{\ 0\leq u \leq (\sigma^2_\xi/ \varrho+\varepsilon)a_T} \hspace{-16pt} \beta_T\abs{W_1(t+u)-W_1(t)}\\
&= \sup_{0\leq t\leq \tilde{T}_\varepsilon-\tilde{a}_{T,\varepsilon}}\sup_{\ 0\leq u \leq \tilde{a}_{T,\varepsilon}} \beta_T\abs{W_1(t+u)-W_1(t)},
\end{align*}
where $\tilde{T}_\varepsilon$ and $\tilde{a}_{T,\varepsilon}$ are defined as $(\sigma^2_\xi/ \varrho+\varepsilon)T$ and $(\sigma^2_\xi/ \varrho+\varepsilon)a_T$ respectively. Introduce 
$$\tilde{\beta}_{T,\varepsilon}:=\left(2\tilde{a}_{T,\varepsilon}\left[\log \frac{\tilde{T}_\varepsilon}{\tilde{a}_{T,\varepsilon}}+\log\log \tilde{T}_\varepsilon \right]\right)^{-1/2},$$
then by Theorem \ref{Brownie_fluctie} we have that 
\begin{equation*}
  \limsup_{T\rightarrow \infty}\sup_{0\leq t\leq \tilde{T}_\varepsilon-a_{T,\varepsilon}}\sup_{0\leq u \leq a_{T,\varepsilon}}\tilde{\beta}_{T,\varepsilon}\abs{W(t+u)-W_t}={\sigma^2_\xi}/{ \varrho} \quad \textrm{a.s.}
\end{equation*}
Similarly, it can be shown that $\limsup A_2=0$ almost surely, which completes the proof. \\
\end{proof}
\section*{Acknowledgements}
This work is part of the research programme ‘Zigzagging through computational barriers’ with project number 016.Vidi.189.043, which is financed by the Dutch Research Council (NWO).

\bibliography{References}

\begin{thebibliography}{84}
\providecommand{\natexlab}[1]{#1}
\providecommand{\url}[1]{\texttt{#1}}
\expandafter\ifx\csname urlstyle\endcsname\relax
  \providecommand{\doi}[1]{doi: #1}\else
  \providecommand{\doi}{doi: \begingroup \urlstyle{rm}\Url}\fi

\bibitem[Andrieu et~al.(2021)Andrieu, Dobson, and
  Wang]{andrieu2021subgeometric}
C.~Andrieu, P.~Dobson, and A.~Q. Wang.
\newblock Subgeometric hypocoercivity for piecewise-deterministic markov
  process monte carlo methods.
\newblock \emph{Electronic Journal of Probability}, 26:\penalty0 1--26, 2021.

\bibitem[Athreya and Ney(1978)]{athreya}
K.~B. Athreya and P.~Ney.
\newblock A new approach to the limit theory of recurrent markov chains.
\newblock \emph{Transactions of the American Mathematical Society},
  245:\penalty0 493--501, 1978.

\bibitem[Bednorz and {\L}atuszy{\'n}ski(2007)]{remarks_fixed}
W.~Bednorz and K.~{\L}atuszy{\'n}ski.
\newblock A few remarks on “fixed-width output analysis for markov chain
  monte carlo” by jones et al.
\newblock \emph{Journal of the American Statistical Association}, 102\penalty0
  (480):\penalty0 1485--1486, 2007.

\bibitem[Berkes et~al.(2011)Berkes, H{\"o}rmann, Schauer, et~al.]{SPLIT_SIP}
I.~Berkes, S.~H{\"o}rmann, J.~Schauer, et~al.
\newblock Split invariance principles for stationary processes.
\newblock \emph{The Annals of Probability}, 39\penalty0 (6):\penalty0
  2441--2473, 2011.

\bibitem[Berkes et~al.(2014)Berkes, Liu, and Wu]{berkes}
I.~Berkes, W.~Liu, and W.~B. Wu.
\newblock Koml{\'o}s--major--tusn{\'a}dy approximation under dependence.
\newblock \emph{The Annals of Probability}, 42\penalty0 (2):\penalty0 794--817,
  2014.

\bibitem[Bhattacharya(1978)]{bhattacharya}
R.~Bhattacharya.
\newblock Criteria for recurrence and existence of invariant measures for
  multidimensional diffusions.
\newblock \emph{The Annals of Probability}, pages 541--553, 1978.

\bibitem[Bierkens and Duncan(2017)]{bierkens2017limit}
J.~Bierkens and A.~Duncan.
\newblock Limit theorems for the zig-zag process.
\newblock \emph{Advances in Applied Probability}, 49\penalty0 (3):\penalty0
  791--825, 2017.

\bibitem[Bierkens and Roberts(2017)]{bierkens_intro}
J.~Bierkens and G.~Roberts.
\newblock A piecewise deterministic scaling limit of lifted
  metropolis--hastings in the curie-weiss model.
\newblock \emph{The Annals of Applied Probability}, 27\penalty0 (2):\penalty0
  846--882, 2017.

\bibitem[Bierkens et~al.(2019{\natexlab{a}})Bierkens, Fearnhead, Roberts,
  et~al.]{zigzagsub}
J.~Bierkens, P.~Fearnhead, G.~Roberts, et~al.
\newblock The zig-zag process and super-efficient sampling for bayesian
  analysis of big data.
\newblock \emph{The Annals of Statistics}, 47\penalty0 (3):\penalty0
  1288--1320, 2019{\natexlab{a}}.

\bibitem[Bierkens et~al.(2019{\natexlab{b}})Bierkens, Roberts, and
  Zitt]{ZZ_ergodicity}
J.~Bierkens, G.~O. Roberts, and P.-A. Zitt.
\newblock Ergodicity of the zigzag process.
\newblock \emph{Ann. Appl. Probab.}, 29\penalty0 (4):\penalty0 2266--2301, 08
  2019{\natexlab{b}}.
\newblock \doi{10.1214/18-AAP1453}.
\newblock URL \url{https://doi.org/10.1214/18-AAP1453}.

\bibitem[Bouchard-C{\^o}t{\'e} et~al.(2018)Bouchard-C{\^o}t{\'e}, Vollmer, and
  Doucet]{bouchard2018}
A.~Bouchard-C{\^o}t{\'e}, S.~J. Vollmer, and A.~Doucet.
\newblock The bouncy particle sampler: A nonreversible rejection-free markov
  chain monte carlo method.
\newblock \emph{Journal of the American Statistical Association}, 113\penalty0
  (522):\penalty0 855--867, 2018.

\bibitem[Bradley(2005)]{Bradley}
R.~C. Bradley.
\newblock Basic properties of strong mixing conditions. a survey and some open
  questions.
\newblock \emph{Probability Surveys}, 2:\penalty0 107--144, 2005.

\bibitem[Cattiaux et~al.(2011)Cattiaux, Chafai, and Guillin]{cattiaux}
P.~Cattiaux, D.~Chafai, and A.~Guillin.
\newblock Central limit theorems for additive functionals of ergodic markov
  diffusions processes.
\newblock \emph{arXiv preprint arXiv:1104.2198}, 2011.

\bibitem[Chakraborty et~al.(2019)Chakraborty, Bhattacharya, and
  Khare]{chakraborty2019}
S.~Chakraborty, S.~K. Bhattacharya, and K.~Khare.
\newblock Estimating accuracy of the mcmc variance estimator: a central limit
  theorem for batch means estimators.
\newblock \emph{arXiv preprint arXiv:1911.00915}, 2019.

\bibitem[Chien et~al.(1997)Chien, Goldsman, and Melamed]{chien1997large}
C.~Chien, D.~Goldsman, and B.~Melamed.
\newblock Large-sample results for batch means.
\newblock \emph{Management Science}, 43\penalty0 (9):\penalty0 1288--1295,
  1997.

\bibitem[Chien(1988)]{chien1988small}
C.-H. Chien.
\newblock Small-sample theory for steady state confidence intervals.
\newblock In \emph{Proceedings of the 20th conference on Winter simulation},
  pages 408--413, 1988.

\bibitem[Cs{\'a}ki and Cs{\"o}rg{\H{o}}(1995)]{sip_regenerative}
E.~Cs{\'a}ki and M.~Cs{\"o}rg{\H{o}}.
\newblock On additive functionals of markov chains.
\newblock \emph{Journal of Theoretical Probability}, 8\penalty0 (4):\penalty0
  905--919, 1995.

\bibitem[Cs{\"o}rg{\"o} and Horv{\'a}th(1993)]{weighted_approx}
M.~Cs{\"o}rg{\"o} and L.~Horv{\'a}th.
\newblock \emph{Weighted approximations in probability and statistics}.
\newblock J. Wiley \& Sons, 1993.

\bibitem[Cs{\"o}rg{\"o} and R{\'e}v{\'e}sz(1979)]{biga_csorgo}
M.~Cs{\"o}rg{\"o} and P.~R{\'e}v{\'e}sz.
\newblock How big are the increments of a wiener process?
\newblock \emph{The Annals of Probability}, pages 731--737, 1979.

\bibitem[Cs{\"o}rg{\"o} and R{\'e}v{\'e}sz(2014)]{sip_boek}
M.~Cs{\"o}rg{\"o} and P.~R{\'e}v{\'e}sz.
\newblock \emph{Strong approximations in probability and statistics}.
\newblock Academic press, 2014.

\bibitem[Cs{\"o}rg{\"o} and Hall(1984)]{applications_sip}
S.~Cs{\"o}rg{\"o} and P.~Hall.
\newblock The koml{\'o}s-major-tusn{\'a}dy approximations and their
  applications.
\newblock \emph{Australian Journal of Statistics}, 26\penalty0 (2):\penalty0
  189--218, 1984.

\bibitem[Cuny et~al.(2018)Cuny, Dedecker, and Merlev{\`e}de]{cuny}
C.~Cuny, J.~Dedecker, and F.~Merlev{\`e}de.
\newblock On the koml{\'o}s, major and tusn{\'a}dy strong approximation for
  some classes of random iterates.
\newblock \emph{Stochastic Processes and their Applications}, 128\penalty0
  (4):\penalty0 1347--1385, 2018.

\bibitem[Damerdji(1991)]{damerdji1991}
H.~Damerdji.
\newblock Strong consistency and other properties of the spectral variance
  estimator.
\newblock \emph{Management Science}, 37\penalty0 (11):\penalty0 1424--1440,
  1991.

\bibitem[Damerdji(1994)]{damerdji1994}
H.~Damerdji.
\newblock Strong consistency of the variance estimator in steady-state
  simulation output analysis.
\newblock \emph{Mathematics of Operations Research}, 19\penalty0 (2):\penalty0
  494--512, 1994.

\bibitem[Damerdji(1995)]{damerdji1995}
H.~Damerdji.
\newblock Mean-square consistency of the variance estimator in steady-state
  simulation output analysis.
\newblock \emph{Operations Research}, 43\penalty0 (2):\penalty0 282--291, 1995.

\bibitem[Davydov(1968)]{Davydov}
Y.~A. Davydov.
\newblock Convergence of distributions generated by stationary stochastic
  processes.
\newblock \emph{Theory of Probability \& Its Applications}, 13\penalty0
  (4):\penalty0 691--696, 1968.

\bibitem[Deheuvels(2000)]{deheuvels}
P.~Deheuvels.
\newblock Uniform limit laws for kernel density estimators on possibly
  unbounded intervals.
\newblock In \emph{Recent advances in reliability theory}, pages 477--492.
  Springer, 2000.

\bibitem[Deligiannidis et~al.(2019)Deligiannidis, Bouchard-C{\^o}t{\'e},
  Doucet, et~al.]{exp_ergo_bouncy}
G.~Deligiannidis, A.~Bouchard-C{\^o}t{\'e}, A.~Doucet, et~al.
\newblock Exponential ergodicity of the bouncy particle sampler.
\newblock \emph{The Annals of Statistics}, 47\penalty0 (3):\penalty0
  1268--1287, 2019.

\bibitem[Douc et~al.(2018)Douc, Moulines, Priouret, and
  Soulier]{douc2018markov}
R.~Douc, E.~Moulines, P.~Priouret, and P.~Soulier.
\newblock \emph{Markov chains}.
\newblock Springer, 2018.

\bibitem[Down et~al.(1995)Down, Meyn, and Tweedie]{down_exponential}
D.~Down, S.~P. Meyn, and R.~L. Tweedie.
\newblock Exponential and uniform ergodicity of markov processes.
\newblock \emph{The Annals of Probability}, pages 1671--1691, 1995.

\bibitem[Duncan et~al.(2016)Duncan, Lelievre, and
  Pavliotis]{duncan2016variance}
A.~B. Duncan, T.~Lelievre, and G.~Pavliotis.
\newblock Variance reduction using nonreversible langevin samplers.
\newblock \emph{Journal of statistical physics}, 163\penalty0 (3):\penalty0
  457--491, 2016.

\bibitem[Durmus et~al.(2020)Durmus, Guillin, Monmarch{\'e}, et~al.]{durmus2020}
A.~Durmus, A.~Guillin, P.~Monmarch{\'e}, et~al.
\newblock Geometric ergodicity of the bouncy particle sampler.
\newblock \emph{Annals of Applied Probability}, 30\penalty0 (5):\penalty0
  2069--2098, 2020.

\bibitem[Einmahl(1989)]{kmt_multi}
Einmahl.
\newblock Extensions of results of koml{\'o}s, major, and tusn{\'a}dy to the
  multivariate case.
\newblock \emph{Journal of multivariate analysis}, 28\penalty0 (1):\penalty0
  20--68, 1989.

\bibitem[El-Nouty(1999)]{fractional_increments}
C.~El-Nouty.
\newblock On the large increments of fractional brownian motion.
\newblock \emph{Statistics \& probability letters}, 41\penalty0 (2):\penalty0
  169--178, 1999.

\bibitem[Fearnhead et~al.(2018)Fearnhead, Bierkens, Pollock, Roberts,
  et~al.]{pdmp_intro}
P.~Fearnhead, J.~Bierkens, M.~Pollock, G.~O. Roberts, et~al.
\newblock Piecewise deterministic markov processes for continuous-time monte
  carlo.
\newblock \emph{Statistical Science}, 33\penalty0 (3):\penalty0 386--412, 2018.

\bibitem[Flegal and Jones(2010)]{flegal_bm}
J.~M. Flegal and G.~L. Jones.
\newblock Batch means and spectral variance estimators in markov chain monte
  carlo.
\newblock \emph{The Annals of Statistics}, 38\penalty0 (2):\penalty0
  1034--1070, 2010.

\bibitem[Fort and Roberts(2005)]{fort_subgeometric}
G.~Fort and G.~O. Roberts.
\newblock Subgeometric ergodicity of strong markov processes.
\newblock \emph{The Annals of Applied Probability}, 15\penalty0 (2):\penalty0
  1565--1589, 2005.

\bibitem[Glynn and Whitt(1991)]{glynn1991}
P.~W. Glynn and W.~Whitt.
\newblock Estimating the asymptotic variance with batch means.
\newblock \emph{Operations Research Letters}, 10\penalty0 (8):\penalty0
  431--435, 1991.

\bibitem[Glynn and Whitt(1992)]{glynn}
P.~W. Glynn and W.~Whitt.
\newblock The asymptotic validity of sequential stopping rules for stochastic
  simulations.
\newblock \emph{The Annals of Applied Probability}, 2\penalty0 (1):\penalty0
  180--198, 1992.

\bibitem[Goldsman et~al.(1990)Goldsman, Meketon, and Schruben]{goldsman1990}
D.~Goldsman, M.~Meketon, and L.~Schruben.
\newblock Properties of standardized time series weighted area variance
  estimators.
\newblock \emph{Management Science}, 36\penalty0 (5):\penalty0 602--612, 1990.

\bibitem[Gong and Flegal(2016)]{gong2016}
L.~Gong and J.~M. Flegal.
\newblock A practical sequential stopping rule for high-dimensional markov
  chain monte carlo.
\newblock \emph{Journal of Computational and Graphical Statistics}, 25\penalty0
  (3):\penalty0 684--700, 2016.

\bibitem[Heunis(2003)]{sip_diff1}
A.~J. Heunis.
\newblock Strong invariance principle for singular diffusions.
\newblock \emph{Stochastic processes and their applications}, 104\penalty0
  (1):\penalty0 57--80, 2003.

\bibitem[Hobert et~al.(2002)Hobert, Jones, Presnell, and Rosenthal]{hobert}
J.~P. Hobert, G.~L. Jones, B.~Presnell, and J.~S. Rosenthal.
\newblock On the applicability of regenerative simulation in markov chain monte
  carlo.
\newblock \emph{Biometrika}, 89\penalty0 (4):\penalty0 731--743, 2002.

\bibitem[H{\"o}pfner and L{\"o}cherbach(2003)]{hopfner}
R.~H{\"o}pfner and E.~L{\"o}cherbach.
\newblock \emph{Limit theorems for null recurrent Markov processes}.
\newblock American Mathematical Soc., 2003.

\bibitem[Hwang et~al.(1993)Hwang, Hwang-Ma, and Sheu]{hwang1993}
C.-R. Hwang, S.-Y. Hwang-Ma, and S.-J. Sheu.
\newblock Accelerating gaussian diffusions.
\newblock \emph{The Annals of Applied Probability}, pages 897--913, 1993.

\bibitem[Jones et~al.(2006)Jones, Haran, Caffo, and Neath]{jones_fixed}
G.~L. Jones, M.~Haran, B.~S. Caffo, and R.~Neath.
\newblock Fixed-width output analysis for markov chain monte carlo.
\newblock \emph{Journal of the American Statistical Association}, 101\penalty0
  (476):\penalty0 1537--1547, 2006.

\bibitem[Kallenberg(1997)]{kallenberg}
O.~Kallenberg.
\newblock \emph{Foundations of modern probability}, volume~2.
\newblock Springer, 1997.

\bibitem[Koml{\'o}s et~al.(1976)Koml{\'o}s, Major, and Tusn{\'a}dy]{kmt_2}
J.~Koml{\'o}s, P.~Major, and G.~Tusn{\'a}dy.
\newblock An approximation of partial sums of independent rv's, and the sample
  df. ii.
\newblock \emph{Zeitschrift f{\"u}r Wahrscheinlichkeitstheorie und verwandte
  Gebiete}, 34\penalty0 (1):\penalty0 33--58, 1976.

\bibitem[Komlos(1975)]{kmt_1}
M.~Komlos.
\newblock Tusnady (1975) an approximation of partial sums of rv’s, and the
  sample df, i, z.
\newblock \emph{Wahrscheinlichkeitstheorie und Verw. Gebiete}, 32:\penalty0
  111--131, 1975.

\bibitem[Kuelbs and Philipp(1980)]{MIX_SIP}
J.~Kuelbs and W.~Philipp.
\newblock Almost sure invariance principles for partial sums of mixing b-valued
  random variables.
\newblock \emph{The Annals of Probability}, pages 1003--1036, 1980.

\bibitem[Lelievre et~al.(2013)Lelievre, Nier, and Pavliotis]{lelievre2013}
T.~Lelievre, F.~Nier, and G.~A. Pavliotis.
\newblock Optimal non-reversible linear drift for the convergence to
  equilibrium of a diffusion.
\newblock \emph{Journal of Statistical Physics}, 152\penalty0 (2):\penalty0
  237--274, 2013.

\bibitem[Li(1992)]{lilimit}
W.~V. Li.
\newblock Limit theorems for the square integral of brownian motion and its
  increments.
\newblock \emph{Stochastic processes and their applications}, 41\penalty0
  (2):\penalty0 223--239, 1992.

\bibitem[Liu et~al.(2021)Liu, Vats, and Flegal]{liu2021batch}
Y.~Liu, D.~Vats, and J.~M. Flegal.
\newblock Batch size selection for variance estimators in mcmc.
\newblock \emph{Methodology and Computing in Applied Probability}, pages 1--29,
  2021.

\bibitem[L{\"o}cherbach and Loukianova(2008)]{locherbach2008num}
E.~L{\"o}cherbach and D.~Loukianova.
\newblock On nummelin splitting for continuous time harris recurrent markov
  processes and application to kernel estimation for multi-dimensional
  diffusions.
\newblock \emph{Stochastic Processes and their Applications}, 118\penalty0
  (8):\penalty0 1301--1321, 2008.

\bibitem[L{\"o}cherbach and Loukianova(2009)]{locherbach44}
E.~L{\"o}cherbach and D.~Loukianova.
\newblock The law of iterated logarithm for additive functionals and martingale
  additive functionals of harris recurrent markov processes.
\newblock \emph{Stochastic processes and their applications}, 119\penalty0
  (7):\penalty0 2312--2335, 2009.

\bibitem[Major(1979)]{major}
P.~Major.
\newblock An improvement of strassen's invariance principle.
\newblock \emph{The Annals of Probability}, 7\penalty0 (1):\penalty0 55--61,
  1979.

\bibitem[Merlev{\`e}de et~al.(2015)Merlev{\`e}de, Rio, et~al.]{merlevede2015}
F.~Merlev{\`e}de, E.~Rio, et~al.
\newblock Strong approximation for additive functionals of geometrically
  ergodic markov chains.
\newblock \emph{Electronic Journal of Probability}, 20, 2015.

\bibitem[Meyn and Tweedie(1993)]{meyn1993}
S.~P. Meyn and R.~L. Tweedie.
\newblock Stability of markovian processes ii: Continuous-time processes and
  sampled chains.
\newblock \emph{Advances in Applied Probability}, pages 487--517, 1993.

\bibitem[Meyn and Tweedie(2012)]{the_gospel}
S.~P. Meyn and R.~L. Tweedie.
\newblock \emph{Markov chains and stochastic stability}.
\newblock Springer Science \& Business Media, 2012.

\bibitem[Mihalache(2012)]{sip_diff2}
S.~Mihalache.
\newblock Strong approximations and sequential change-point analysis for
  diffusion processes.
\newblock \emph{Statistics \& Probability Letters}, 82\penalty0 (3):\penalty0
  464--472, 2012.

\bibitem[Mykland et~al.(1995)Mykland, Tierney, and Yu]{mykland}
P.~Mykland, L.~Tierney, and B.~Yu.
\newblock Regeneration in markov chain samplers.
\newblock \emph{Journal of the American Statistical Association}, 90\penalty0
  (429):\penalty0 233--241, 1995.

\bibitem[Nummelin(1978)]{nummelin}
E.~Nummelin.
\newblock A splitting technique for harris recurrent markov chains.
\newblock \emph{Zeitschrift f{\"u}r Wahrscheinlichkeitstheorie und verwandte
  Gebiete}, 43\penalty0 (4):\penalty0 309--318, 1978.

\bibitem[Nummelin(2004)]{nummelin_boek}
E.~Nummelin.
\newblock \emph{General irreducible Markov chains and non-negative operators}.
\newblock Cambridge University Press, 2004.

\bibitem[Nummelin and Tuominen(1982)]{nummelin1982}
E.~Nummelin and P.~Tuominen.
\newblock Geometric ergodicity of harris recurrent marcov chains with
  applications to renewal theory.
\newblock \emph{Stochastic Processes and Their Applications}, 12\penalty0
  (2):\penalty0 187--202, 1982.

\bibitem[Nummelin and Tuominen(1983)]{nummelin_poly}
E.~Nummelin and P.~Tuominen.
\newblock The rate of convergence in orey's theorem for harris recurrent markov
  chains with applications to renewal theory.
\newblock \emph{Stochastic Processes and Their Applications}, 15\penalty0
  (3):\penalty0 295--311, 1983.

\bibitem[Ortega(1984)]{ortegasize}
J.~Ortega.
\newblock On the size of the increments of nonstationary gaussian processes.
\newblock \emph{Stochastic processes and their applications}, 18\penalty0
  (1):\penalty0 47--56, 1984.

\bibitem[Parzen(1979)]{parzen}
E.~Parzen.
\newblock Nonparametric statistical data modeling.
\newblock \emph{Journal of the American statistical association}, 74\penalty0
  (365):\penalty0 105--121, 1979.

\bibitem[Philipp and Stout(1975)]{philippSIP}
W.~Philipp and W.~Stout.
\newblock \emph{Almost sure invariance principles for partial sums of weakly
  dependent random variables}, volume 161.
\newblock American Mathematical Soc., 1975.

\bibitem[R{\'e}v{\'e}sz(1982)]{revesz1982}
P.~R{\'e}v{\'e}sz.
\newblock On the increments of wiener and related processes.
\newblock \emph{The Annals of Probability}, pages 613--622, 1982.

\bibitem[Revuz(2008)]{revuz}
D.~Revuz.
\newblock \emph{Markov chains}.
\newblock Elsevier, 2008.

\bibitem[Rey-Bellet and Spiliopoulos(2015)]{rey2015irreversible}
L.~Rey-Bellet and K.~Spiliopoulos.
\newblock Irreversible langevin samplers and variance reduction: a large
  deviations approach.
\newblock \emph{Nonlinearity}, 28\penalty0 (7):\penalty0 2081, 2015.

\bibitem[Rio(1993)]{Rio}
E.~Rio.
\newblock Covariance inequalities for strongly mixing processes.
\newblock In \emph{Annales de l'IHP Probabilit{\'e}s et statistiques},
  volume~29, pages 587--597, 1993.

\bibitem[Roberts and Rosenthal(2001)]{roberts_rosenthal}
G.~O. Roberts and J.~S. Rosenthal.
\newblock Markov chains and de-initializing processes.
\newblock \emph{Scandinavian Journal of Statistics}, 28\penalty0 (3):\penalty0
  489--504, 2001.

\bibitem[Rogers and Williams(2000)]{rogerswilliams}
L.~C.~G. Rogers and D.~Williams.
\newblock \emph{Diffusions, Markov processes and martingales: Volume 2, It{\^o}
  calculus}, volume~2.
\newblock Cambridge university press, 2000.

\bibitem[Sherman and Goldsman(2002)]{sherman2002large}
M.~Sherman and D.~Goldsman.
\newblock Large-sample normality of the batch-means variance estimator.
\newblock \emph{Operations Research Letters}, 30\penalty0 (5):\penalty0
  319--326, 2002.

\bibitem[Shorack and Wellner(2009)]{shorack_empirical}
G.~R. Shorack and J.~A. Wellner.
\newblock \emph{Empirical processes with applications to statistics}.
\newblock SIAM, 2009.

\bibitem[Sigman(1990)]{sigman}
K.~Sigman.
\newblock One-dependent regenerative processes and queues in continuous time.
\newblock \emph{Mathematics of Operations Research}, 15\penalty0 (1):\penalty0
  175--189, 1990.

\bibitem[Song and Schmeiser(1995)]{song1995optimal}
W.~T. Song and B.~W. Schmeiser.
\newblock Optimal mean-squared-error batch sizes.
\newblock \emph{Management Science}, 41\penalty0 (1):\penalty0 110--123, 1995.

\bibitem[Stramer and Tweedie(1999)]{stramer1999}
O.~Stramer and R.~Tweedie.
\newblock Langevin-type models i: Diffusions with given stationary
  distributions and their discretizations.
\newblock \emph{Methodology and Computing in Applied Probability}, 1\penalty0
  (3):\penalty0 283--306, 1999.

\bibitem[Stroock and Varadhan(1997)]{stroock1997}
D.~W. Stroock and S.~S. Varadhan.
\newblock \emph{Multidimensional diffusion processes}, volume 233.
\newblock Springer Science \& Business Media, 1997.

\bibitem[Tweedie(1993)]{tweedie_genres}
S.~Tweedie.
\newblock Generalized resolvents and harris recurrence.
\newblock \emph{Doeblin and modern probability}, 149, 1993.

\bibitem[Vats et~al.(2018)Vats, Flegal, Jones,
  et~al.]{multivariate_consistency}
D.~Vats, J.~M. Flegal, G.~L. Jones, et~al.
\newblock Strong consistency of multivariate spectral variance estimators in
  markov chain monte carlo.
\newblock \emph{Bernoulli}, 24\penalty0 (3):\penalty0 1860--1909, 2018.

\bibitem[Vats et~al.(2019)Vats, Flegal, and Jones]{multivariate_output}
D.~Vats, J.~M. Flegal, and G.~L. Jones.
\newblock Multivariate output analysis for markov chain monte carlo.
\newblock \emph{Biometrika}, 106\penalty0 (2):\penalty0 321--337, 2019.

\bibitem[Williams(2006)]{williams2006}
D.~Williams.
\newblock \emph{Stochastic integrals: proceedings of the LMS Durham Symposium,
  July 7-17, 1980}, volume 851.
\newblock Springer, 2006.

\end{thebibliography}

\end{document}